\documentclass[11pt,a4paper]{article}
\usepackage[utf8]{inputenc}
\usepackage[english]{babel}
\usepackage[T1]{fontenc}
\usepackage{lmodern}
\usepackage{amsmath}
\usepackage{accents}
\usepackage{amsfonts}
\usepackage{amssymb}
\usepackage{amsthm}
\usepackage{amscd}
\usepackage{url}
\usepackage[shortlabels]{enumitem}
\usepackage{hhline}
\usepackage{dsfont}
\usepackage{array}

\usepackage{geometry}
\usepackage{tikz}
\usetikzlibrary{decorations.pathreplacing}
\usetikzlibrary{patterns}
\usepackage{xcolor}
\usetikzlibrary{arrows}
\usepackage{mdframed}
\usepackage{caption}
\usepackage[titletoc]{appendix}
\usepackage{csquotes}
\usepackage{hyperref}
\usepackage{float}
\usepackage{bm}
\usepackage{microtype}
\usepackage{graphicx}

\mdfsetup{nobreak=true}

\AtBeginDocument{
  \label{CorterrectFirstPageLabel}
  
}

\usepackage[
style=alphabetic,
maxalphanames=10,
minalphanames=5,
maxbibnames=10,
firstinits=true,
backend=biber,
doi=false,isbn=false,url=false
]{biblatex}
\newcolumntype{L}{>{$}l<{$}}

\newcommand\f{f} 

\newcommand{\BB}{\mathcal{B}}

\newcommand\N{\mathbb{N}}

\newcommand\HH{\mathcal{H}}

\newcommand\R{\mathbb{R}}

\newcommand\C{\mathbb{C}}

\newcommand\Ss{\mathcal{S}}
\newcommand\Sp{\mathbb{S}}

\newcommand\E{\mathbb{E}}
\newcommand\Id{\mathrm{Id}}
\DeclareMathOperator\Var{Var}
\DeclareMathOperator\supp{supp}

\newcommand\dd{\mathrm{d}}
\newcommand\CC{\mathcal{C}}

\newcommand\GOE{\mathrm{GOE}}

\DeclareMathOperator\Ker{Ker}
\DeclareMathOperator\Sym{Sym}
\DeclareMathOperator\Imm{Im}

\DeclareMathOperator\Sesq{Sesq}

\DeclareMathOperator\Tr{Tr}

\newcommand{\enstq}[2]{\left\{#1~\middle|~#2\right\}}

\newcommand\Vect{\mathrm{Vect}}
\DeclareMathOperator\Vol{Vol}
\DeclareMathOperator\dist{dist}

\newcommand\str{\;\bar{*}\;}
\newcommand\one{\mathds{1}}

\newcommand\numberthis{\addtocounter{equation}{1}\tag{\theequation}}
\newcommand\jump{\par\medskip}
\newcommand\quand{\quad\text{and}\quad}
\newcommand\qsothat{\quad\text{so that}\quad}

\newcommand\qwith{\quad\text{with}\quad}

\theoremstyle{plain}
\newtheorem{theo}{Theorem}[section]
\newenvironment{theorem}%
  {\begin{mdframed}[backgroundcolor=white]\begin{theo}}%
  {\end{theo}\par\vspace{0.1cm}\end{mdframed}}
  
\theoremstyle{plain}
\newtheorem{coro}[theo]{Corollary}
\newenvironment{corollary}%
  {\begin{mdframed}[backgroundcolor=white]\begin{coro}}%
  {\end{coro}\par\smallskip\end{mdframed}}
  
\theoremstyle{plain}
\newtheorem{lemm}[theo]{Lemma}
\newenvironment{lemma}%
  {\begin{mdframed}[backgroundcolor=white]\begin{lemm}}%
  {\end{lemm}\par\vspace{0.cm}\end{mdframed}}

\theoremstyle{plain}
\newtheorem{prop}[theo]{Proposition}
  {\begin{mdframed}[backgroundcolor=white]\begin{prop}}%
  {\end{prop}\par\vspace{0.cm}\end{mdframed}}
  
\theoremstyle{definition}
\newtheorem{defn}[theo]{Definition}
  {\begin{mdframed}[backgroundcolor=white]\begin{defn}}%
  {\end{defn}\par\vspace{0.cm}\end{mdframed}}

\newenvironment{acknowledgements}{%
  \begin{abstract}
}{%
  \end{abstract}
}

\theoremstyle{definition}
\newtheorem{remark}[theo]{Remark}

\makeatletter
\def\blfootnote{\gdef\@thefnmark{}\@footnotetext}
\makeatother
  
\makeatletter
\renewcommand*\env@matrix[1][*\c@MaxMatrixCols c]{%
  \hskip -\arraycolsep
  \let\@ifnextchar\new@ifnextchar
  \array{#1}}
\makeatother

\begin{document}

\title{\Huge{Spectral Criteria for the Asymptotics of Local Functionals of Gaussian Fields and Their Application to Nodal Volumes and Critical Points}}
\author{Louis Gass}
\maketitle
\blfootnote{University of Luxembourg}
\blfootnote{This work was supported by the Luxembourg National Research Fund (Grant: 021/16236290/HDSA)}
\blfootnote{Email: louis.gass(at)uni.lu}

\begin{abstract}
We establish a general criterion for the positivity of the variance of a chaotic component of local functionals of stationary vector-valued Gaussian fields. This criterion is formulated in terms of the spectral properties of the covariance function, without requiring integrability or isotropy. It offers a simple and robust framework for analyzing variance asymptotics in such models. We apply this approach to the study of the nodal volume and the number of critical points of a Gaussian field, proving the positivity of the limiting variance under mild conditions on the covariance function. Additionally, we examine the asymptotics of nodal volume and critical points of Euclidean random waves, deriving the central limit theorem through an analysis of the second and fourth chaotic components. As a byproduct, we unify and generalize many existing results on the volume of intersections of random waves and their critical points, bypassing the need for traditional, intricate variance computations. Our findings shed new light on the second-chaos cancellation phenomenon from a spectral perspective and can be extended to any local, possibly singular, functional of Gaussian fields.
\end{abstract}
\newpage

\renewcommand\contentsname{} 

\begingroup
\let\clearpage\relax
\vspace{-1cm} 
\setcounter{tocdepth}{2}
\tableofcontents
\endgroup

\setcounter{tocdepth}{2}
\tableofcontents\jump

\section{Introduction and main results}
\label{sec10}
\subsection{A general framework}
\label{sec11}
Let $V$ and $W$ be two finite dimensional vector spaces of dimension $d$ and $n$ respectively. Let $X:V\rightarrow W$ be a vector-valued stationary Gaussian random field with a continuous non-degenerate covariance function $\Omega:V\rightarrow S^2(W^*)$, where this last space denotes the set of symmetric bilinear maps on $W$. We denote by $\eta$ the associated Gaussian measure on $W$, and let $f : W\rightarrow \R$ in $L^2(\eta)$. Our primal object of interest is the quantity 
\[Z_\lambda(\phi) = \frac{1}{\lambda^{d/2}}\int_{V}\phi\left(\frac{v}{\lambda}\right)f(X(v))\dd v,\numberthis\label{eq:01}\]
for a test function $\phi$ such that $\|\phi\|_2=1$. We say that $Z_\lambda$ satisfies a Central Limit Theorem (CLT) if the following convergence in distribution holds
\[\frac{Z_\lambda(\phi)-\E[Z_\lambda(\phi)]}{\sqrt{\Var(Z_\lambda(\phi))}}\underset{\lambda \rightarrow +\infty}{\longrightarrow} \mathcal{N}(0,1)\]
A proof of the CLT is usually split into two very distinct problems: the asymptotic normality and the positivity of the limiting variance. The asymptotic normality is well-understood and was first proved in \cite{Cuz76} in the case $d=n=1$ and $\Omega\in L^1(\R^d)$ by approximation using $m$-dependent processes. Since then, it has been refined with the introduction of the Wiener chaos expansion, which has become an indispensable tool for proving asymptotic results related to $Z_\lambda(\phi)$. In particular, the celebrated Breuer--Major theorem \cite{Bre83} states that asymptotic normality holds if the Hermite rank of $f$ is $q$ and $\Omega\in L^q(V)$. Modern proofs of this theorem now include quantitative CLT and functional convergence \cite{Nou11, Nou20, Nua20} and make use of more advanced techniques, such as Malliavin calculus and Stein's method. The literature on the topic is extensive.\jump

In the case $n=1$, the positivity of the limiting variance can be established using simple Fourier arguments under minimal conditions, such as assuming that $f$ is not odd or that the spectral measure has a density with respect to the origin, see \cite{Cuz76}. For $n>1$, the situation is more intricate, and the positivity of the limiting variance cannot be inferred from straightforward observations on $f$ and $\Omega$. The first goal of this paper is to establish general conditions on $\Omega$ and $f$ to ensure that the limiting variance of $Z_\lambda(\phi)$ is positive, i.e., that $Z_\lambda(\phi)$ does not vanish at infinity.\jump

The heuristic of the approach for obtaining a lower bound for the limiting variance is the following. Since the process $f\circ X:V\rightarrow \R$ is stationary, its covariance function is a positive semi-definite function, meaning that its Fourier transform $\rho$ is a finite measure on $V^*$. This fact can in turn be exploited to obtain lower bound for the variance in \eqref{eq:01}. Indeed, in the simple case where $\rho$ has a density near the origin w.r.t the Lebesgue measure then
\[\lim_{\lambda\rightarrow +\infty} \Var(Z_\lambda(\phi)) = \rho(0),\]
thus the limiting variance is positive if and only if $\rho(0)$ is positive. In a more general framework, the variance can be understood from local properties of the spectral measure $\rho$ near the origin. To explicitly compute the measure $\rho$ in terms of the spectral measure $\mu$ associated to $X$, the most direct path is to expand $f$ into its orthogonal Hermite decomposition on the space $L^2(\eta)$. The variance will then be expressed as sum of powers of the covariance function $\Omega$. Passing to Fourier transform, this yields an expression of $\rho$ in terms of a sum of non-negative terms, given by iterated convolutions of the spectral measure $\mu$ and Hermite expansion of $f$. Additional comments on the advantage of this approach will be made after the statement of Theorem \ref{thm4}.\jump

Let $\eta$ be a non-degenerate Gaussian measure on $W$ with associated scalar product $\langle\,\cdot\,,\,\cdot\,\rangle$, and let $(H^q)_{q\geq 0}$ be the associated sequence of Hermite $q$-forms, seen as polynomial functions on $W$ taking values in the space of symmetric $q$-linear maps on $W$. If we identify $W$ with $\R^n$ then $H^q$ coincides with the standard multivariate hermite polynomials. Note that this point of view in terms of multilinear forms, without the use of coordinates, allows for more compact and concise expressions, which will greatly simplify the statement of the theorems and their proofs. A function $f\in L^2(\eta)$ can be expanded into its chaotic coefficients $(f_q)_{\geq 0}$ on the $q$-th Wiener space, seen as symmetric $q$-linear maps on $W$, so that in holds in $L^2(\eta)$ the decomposition
\[f = \sum_{q=0}^{+\infty} \pi_q(f) = \sum_{q=0}^{+\infty}\frac{1}{q!}\langle f_q, H^q\rangle.\]
We refer to Section \ref{sec20} for more details.\jump

Since in the following, we use Fourier transform arguments, but we want to keep the notations to general vector space notations, we introduce $W_\C = W+iW$, the complexification of $W$. One must simply understand $W$ as $\R^n$ and $W_\C$ as $\C^n$. By Bochner theorem, the Fourier transform $\mu$ of the covariance function $\Omega$ is an operator-valued measure on $V^*$, taking values in the space of positive Hermitian form on $W_\C^*$. Its $q$-fold convolution, denoted by $\mu^{*q}$, can be seen as measure on $V$, taking values in the space of positive Hermitian forms on the space of symmetric $q$-linear forms on $W_\C$. These facts are explained in more detail in Section \ref{sec30}.\jump

In the following, we fix a function $\phi\in L^1\cap L^2(V)$, with $\|\phi\|_2=1$, and we define $\gamma = |\hat{\phi}|^2$ on $V^*$, where the $\hat{}$ notation stand for the Fourier transform. Let $\gamma_\lambda:v\mapsto \lambda^d\gamma(\lambda v)$, which weakly asymptotically converges to the Dirac delta function at $0$, since by Plancherel isometry, $\|\gamma\|_1=1$. We define the least radially decreasing majorant and the greatest radially decreasing minorant of $\gamma$ by the formulas
\[\gamma^+ :x\rightarrow \sup_{\|y\|\geq \|x\|}\gamma(y)\quand \gamma^- :x\rightarrow \inf_{\|y\|\leq \|x\|}\gamma(y).\]
We say that $\phi$ is a \textit{test function} if $\gamma^+\in L^1(V^*)$. By property of the Fourier transform, if the function $\phi$ is absolutely continuous with bounded variation, or if it is the Fourier transform of a ball or more generally of a smooth compact hypersurface with positive Gauss curvature, this hypothesis is satisfied, see \cite{Her62} and the heuristic explanation before Lemma \ref{lemma16}. Note also that $\gamma$ is continuous and $\gamma(0) = 1$, so that $\gamma^-$ is nonzero.\jump

\subsubsection{A general spectral expression for the variance}
We define the $q$-th chaotic projection of $Z_\lambda(\phi)$ by
\[Z_\lambda^{(q)}(\phi) = \frac{1}{\lambda^{d/2}}\int_{V}\phi\left(\frac{v}{\lambda}\right)\pi_q(f)(X(v))\dd v,\]
\begin{theorem}
\label{thm1}
One has
\[\Var(Z_\lambda(\phi)) = \sum_{q=1}^{+\infty} \Var(Z_\lambda^{(q)}(\phi)).\]
and
\[\Var(Z_\lambda^{(q)}(\phi))= \frac{1}{q!}\int_{V^*}\gamma_\lambda(\xi)\,\dd \mu^{*q}(f_q)(\xi).\]
\end{theorem}

Since $\mu^{*q}$ is a positive Hermitian measure, the measure $\mu^{*q}(f_q)$ is non-negative. The expression of the variance in \ref{thm1} then has the advantage, compared to the "non-spectral" expression for the variance of the $q$-th chaos, of being written as the integrand of a non-negative measure, which is our first starting point to give lower bound for the limiting variance. As $\lambda$ goes to infinity, this expression formally converges to the Lebesgue density at $0$ of the measure $\mu^{*q}(f_q)$. To precise this heuristic, we define for a real bounded measure $\nu$ on $V^*$ the quantities
\[D^+\nu = \limsup_{R\rightarrow 0}\frac{\nu(B(0,R))}{\Vol(B(0,R))} \quand D^-\nu = \liminf_{R\rightarrow 0}\frac{\nu(B(0,R))}{\Vol(B(0,R))}.\]
If $D^-\nu = D^+\nu$, we call this quantity $D\nu$. The quantity $D\nu$, when it exists, can be thought as a weak version of the density at $0$ of the measure $\nu$ w.r.t the Lebesgue measure on $V^*$.
\begin{corollary}
\label{cor1}
\[\liminf_{\lambda\rightarrow +\infty} \Var(Z_\lambda^{(q)}(\phi))\geq  \frac{1}{q!}\|\gamma^-\|_1D^-(\mu^{*q}(f_q)),\]
and
\[\limsup_{\lambda\rightarrow +\infty} \Var(Z_\lambda^{(q)}(\phi))\leq \frac{1}{q!}\|\gamma^+\|_1 D^+(\mu^{*q}(f_q)).\]
\end{corollary}

This corollary, which is a direct corollary of Theorem \ref{thm1} and Lemma \ref{lemma7}, puts forward the interaction between the iterated convolutions of the spectral measure, their regularity near the origin, and asymptotic results of local functionals of the Gaussian field. In the case $n=1$, the $q$-fold convolution of $\mu$ can be seen as the distribution of a random walk with transition probability distribution $\mu$ after $q$ steps. The properties of $\mu^{*q}$ near the origin are linked to the probability of return around $0$ of the random walk. This analogy has been successfully used in the series of papers \cite{Mar08, Mar10}, motivated  by the probabilistic modeling of the cosmic microwave background radiation, more recently in \cite{Lac22} to obtain variance asymptotics related to the volume of random excursions sets, and in \cite{Gro24} to study the covariance structure of the Berry random wave model. 
\jump

If $D\nu$ exists and is finite, one can conclude for general a measure $\mu$ that the variance converges if the function $\gamma$ is radially decreasing. If one assumes more hypothesis on $\mu^{*q}$, for instance that it has a bounded continuous density, then the convergence follows from more classical results concerning convolution. This is for instance the case if the covariance function $\Omega$ is in $L^q$, which is the classical setting of the Breuer-Major theorem, and one can drop the assumption that $\gamma^+$ is integrable. The following theorem gives alternative expressions in the case where the spectral measure $\mu$ has a density in $L^p$, where $p$ is the conjugate exponent of $q$.
\begin{corollary}
\label{cor2}
Let $p$ be the conjugate exponent to $q$. If $\mu$ has a density $\Sigma\in L^p(V^*)$ (in $C_0(V^*)$ if $q=1$) w.r.t the Lebesgue measure, then
\[\lim_{\lambda\rightarrow +\infty}\Var(Z_\lambda^{(q)}(\phi)) = \frac{1}{q!}\int_{\xi_1+\ldots+\xi_q=0}\bigotimes_{k=1}^q\Sigma(\xi_k)(f_q)\dd\xi.\]
For $q=1$, 
\[\lim_{\lambda\rightarrow +\infty}\Var(Z_\lambda^{(1)}(\phi)) = \Sigma(0)(f_1).\]
For $q=2$, in matrix form,
\[\lim_{\lambda\rightarrow +\infty}\Var(Z_\lambda^{(2)}(\phi)) = \frac{1}{2}\int_{V^*} \Tr(\Sigma(\xi)f_2\Sigma(\xi)f_2)\dd\xi.\]
\end{corollary}

\subsubsection{A general cancellation theorem for the variance}
If $n=1$ then $f_q$ is a real number and $\mu$ is a bounded measure on $V^*$, so that $\mu^{*q}(f_q) = f_q^2\mu^{*q}$. At a fixed chaos, if $f_q$ is nonzero, the variance asymptotics is solely contained in the spectral measure $\mu$. This observation is substantially wrong as soon as the dimension $n$ is greater than one. It can happen that the variance asymptotics greatly differs for two different non zero $q$-linear maps $f_q$ and $g_q$. This general fact is generally described as a \textit{cancellation phenomenon}. This phenomenon is remarkable in the sense that a cancellation phenomenon can happen only for a class of $q$-linear map of codimension at least one, i.e. almost never. \jump

Several examples of cancellation phenomenon are issued from the context of random waves, see Section \ref{sec13}, but let us give a simple one here. Let $Y(t)$ be a random Gaussian process with covariance function $r(t) = \sin(t)/t$. We define the Gaussian process $X:\R\rightarrow \R^2$ by $X(t) = (Y(t), Y(t+\pi))$, with covariance function $\Omega$. Let $f,g:\R^2\rightarrow \R$ be defined by 
\[f(x,y) = (x^2-1) - (y^2-1)\quand g(x,y) = (x^2-1)+(y^2-1).\]
Both $f$ and $g$ belongs to the second Wiener chaos, and a direct computation shows that
\[\lim_{\lambda\rightarrow +\infty} \Var\left(\frac{1}{\sqrt{\lambda}}\int_0^{\lambda} f(X(t))\dd t\right) = 0\quad\text{but}\quad \lim_{\lambda\rightarrow +\infty} \Var\left(\frac{1}{\sqrt{\lambda}}\int_0^{\lambda} g(X(t))\dd t\right) =4\|r\|_2^2>0,\numberthis\label{eq:03}\]
so that the random process $Y$ is subject to cancellation. \jump

For a Hermitian form $f$ on a complex vector space $W$, we define its \textit{isotropic cone}
\[C(f) = \enstq{w\in W}{f(w)=0}.\]
The integrand in the limiting variance Corollary \ref{cor2} is non-negative. If there is cancellation, then the integrand is null for almost all coordinates $\xi_1,\ldots,\xi_q$ with null barycenter, which is a very restrictive condition. For instance, in the case $q=1$, the limiting variance cancels if and only if $f\in \Ker(\Sigma(0))$. In the case $q=2$, one can see (see Lemma \ref{lemma4} for a proof), that the limiting variance cancels if and only if a.s. $\Imm(\Sigma(\xi))\subset C(f_2)$. More generally, for $q$ even, one can exploit the fact that $\mu^{*2q} = (\mu^{*q})^{*2}$, so that formally, the Lebesgue density at zero of $\mu^{*2q}$ can be thought as the $L^2$ norm of $\mu^{*q}$. We define one more concept before stating the last theorem of this section.\jump

Given any measure $\mu$ taking values in the space of positive Hermitian forms on $W$, one can find a real measure $\nu$ on $V^*$ such that the measure $\mu$ has a density $\Sigma$ w.r.t to $\nu$, where $\Sigma$ is defined $\nu$-a.s. as a mapping from $V$ to the space of positive Hermitian forms on $w$ (a canonical choice for $\nu$ is the trace of $\mu$). The writing $\dd\mu = \Sigma \dd\nu$ will be called a \textit{representation} of $\mu$. One can thus define, the \textit{image} of $\mu$, denoted by $\Imm(\mu)$, as the smallest closed subset $A\subset W^*$ such that
\[\nu\left(\enstq{x\in V}{\Imm(\Sigma(x))\not\subset A}\right)=0.\]
This notion coincides with the essential range of the measurable function $(x,w)\mapsto \Sigma(x)(w,.)$, and is independent on a choice of representation of $\mu$, see Lemma \ref{lemma6}. For $q\geq 0$, the $2q$-th chaotic projection $f_{2q}$ can be seen as an Hermitian form on the dual of symmetric $q$-linear form on $W$, so that the spaces $\Imm(\mu^{*q})$ and $C(f_{2q})$ are both subsets of the dual of the space of symmetric $q$-linear map on $W_\C$.
\begin{theorem}
\label{thm4}
\strut\\ 
\begin{itemize}
\item If $\Imm(\mu^{*q})\not\subset C(f_{2q})$ then
\[\liminf_{\lambda\rightarrow +\infty} \Var(Z_\lambda^{(2q)}(\phi))>0.\]
\item If $\Imm(\mu^{*q})\subset C(f_{2q})$, and either
\begin{itemize}
\item $\Omega\in L^{2q}(V)$
\item $\mu^{*q}$ has a $\CC^1$ density w.r.t the Lebesgue measure of a smooth compact hypersurface,
\end{itemize}
then
\[\lim_{\lambda\rightarrow +\infty} \Var(Z_\lambda^{(2q)}(\phi))=0.\]
\end{itemize}
\end{theorem}
Let $\mu^{*q} = \Sigma_q\dd\nu$ be a representation of $\mu^{*q}$. From Lemma \ref{lemma3}, one has the equivalence
\[\Imm(\mu^{*q}) \subset C(f_{2q})\quad\Longleftrightarrow\quad \nu-\text{a.s.},\;\; \Sigma_q^{\otimes 2}(f_{2q},f_{2q}) = 0,\]
and we will use both forms in applications of Theorem \ref{thm4}. In the previous example \eqref{eq:03}, the image of the Fourier transform of the covariance matrix $\Omega$ is given by
\[\Sigma(\xi) = \begin{pmatrix}
\hat{r}(\xi) & \hat{r}(\xi)e^{i\pi \xi} \\ 
\hat{r}(\xi)e^{-i\pi \xi} & \hat{r}(\xi)
\end{pmatrix}\quand \Imm(\Sigma) = \enstq{\alpha\begin{pmatrix}
1 \\ 
e^{-i\pi\xi}\end{pmatrix}}{\alpha\in \C, \xi\in\R},\]
which included in the isotropic cone of $f$ equal to $C(f) = \enstq{(z_1,z_2)\in \C^2}{|z_1|=|z_2|}$, but not in that of $g$, which is reduced to $\{0\}$. \jump

Let us give a short description, in parallel with related works, of the possible behavior of the variance as $\lambda$ goes to infinity of the total variance of $Z_\lambda(\phi)$. In the following, we set $q$ to be the Hermite rank of $f$, that is the smallest non-zero index such that $f_q$ is nonzero.\jump

\noindent \underline{-- If $\Omega\in L^q(V)$ :}\jump

\noindent  
\begin{itemize}
\item[$\bullet$] If $q$ is even and $\Imm(\mu^{*q/2})\not\subset C(f_q)$, then Theorem \ref{thm4} is an equivalence and the limiting variance is positive. By the standard contraction principle and Arcones inequality, the variance converges and the CLT holds for $Z_\lambda(\phi)$.
\item[$\bullet$] If $q$ is even and $\Imm(\mu^{*q/2})\subset C(f_q)$ then the limiting variance cancels. One must look at the next non-zero chaotic component, if it exists, and prove the positivity of the limit variance, either by the same theorem applied with a greater $q$ or by other mean.
\item[$\bullet$] If $q$ is odd, one must use the general corollary \ref{cor1} or the explicit representation \ref{cor2} if in addition $\mu\in L^p$ (this not automatic if $q\geq 3$), to prove a positive lower bound for the variance, from which would follow a CLT.
\end{itemize}

\noindent \underline{-- If $\Omega\notin L^q(V)$ :}\jump

\noindent In that case, the situation is delicate, and several behaviors may occur. 
\begin{itemize}
\item[$\bullet$] In the case $W=\R$, if $\Omega$ is oscillating, meaning that it is not square integrable but the spectral measure has a singularity near the origin which is sufficiently integrable, then it has been proved in \cite{Mai24} that the $q$-th chaos to dominates and the model tends to exhibit a CLT with a variance of higher order than in the case $\Omega\in L^q$. 
\item[$\bullet$] On the contrary, if $\Omega$ slowly varying, implying in particular that the spectral measure has a singularity at $0$ that explodes with a sufficiently bad rate, then the $q$-th chaos again tends to dominate, with a variance growing at a higher order than the case $\Omega\in L^q$, but this time exhibits an explicit non-CLT behavior, see for instance \cite{Dob79}, and the functional limit is not a Brownian motion. Note that a unifying spectral condition of these two last two points is still missing in the literature, and it is not clear whether the distinction between them is solely contained in the singularity near zero of the spectral measure. One direction toward such result, would be, in light of the fourth moment Theorem \cite{Nua05} and the contraction principle, to write the contractions in term of the spectral measure and observe their behaviors in the limit. These phenomenon are likely to have a multidimensional counterpart.
\item[$\bullet$] If $q$ is even and $\Imm(\mu^{*q/2})\subset C(f_q)$, there is a competition between the cancellation phenomenon and the higher growth of the variance explained in the first two points. From Theorem \ref{thm4}, if the inclusion is sufficiently "smooth" then the $q$-th chaos cancels. It might not be the case for more irregular measures, where one can only rely on the general Corollary \ref{cor2}. The random waves model, whose spectral measure is supported on the sphere, is a typical model susceptible to fall in this scenario. To obtain the variance asymptotics if the cancellation occurs, one must look at the next non-zero chaotic component and see if it falls in one of the previous aforementioned situations. In any case, the variance asymptotic will be of lower order than the one expected from the heuristic of the case $W=\R$. We refer to Section \ref{sec14} for more details.
\end{itemize}

\subsection{Nodal volume and critical points of Gaussian fields}
\label{sec12}
Let $V$ and $U$ be two finite dimensional Euclidean spaces of dimension $d,k$ respectively, with $k\leq d$. Let $Y:V\rightarrow U$ be a non-degenerate stationary Gaussian random field with covariance function $r:V\rightarrow U\otimes U$, and spectral measure $\psi$. We assume that 
\begin{itemize}
\item $Y$ is a $\CC^2$ random field
\item For $v,w\in V$ distinct, the vector $(Y(v),Y(w))$ is non-degenerate
\item $r$ decreases to zero at infinity
\end{itemize}
It follows from this set of assumption that for $u\in U$, the level set $\{Y=u\}$ is generically a smooth submanifold of codimension $k$ in $V$, to which we associate its random nodal measure $\dd Z^u$. If $u=0$ we will denote it by $\dd Z$. We are interested, in the quantity
\[Z_\lambda^u(\phi) = \frac{1}{\lambda^{d/2}}\int_V \phi\left(\frac{v}{\lambda}\right)\dd Z^u(v).\]
When $U=V$, $u=0$ and $Y = \nabla F$ for a real random field $F$ on $V$, then $dZ$ coincide with the Dirac measure supported on the set of critical points of the random field $F$.\jump

The study of the number of zeros of a random function goes back to Kac and Rice \cite{Kac43, Ric45}, and has since been intensively studied in the past decades. By means of the Kac--Rice formula, the quantity $Z_\lambda^u(\phi)$ can be formally included in the general framework \eqref{eq:01}, where the function $f$ is seen as a tempered distribution, see Section \ref{sec23}, and $X = (Y, \nabla Y)$. It is then not surprising that the study of the nodal volume uses similar tools as those employed in the more general context of Section \ref{sec11}.\jump

The CLT for the number of zeros on a growing interval of a stationary Gaussian process with $L^2$ covariance function and its second derivative was proved in \cite{Cuz76} using the method of approximation by $m$-dependent processes, under the assumption of positivity of the limiting variance. This result was extended, under a similar square integrability condition, to the nodal length of a Gaussian field from $\mathbb{R}^2$ to $\mathbb{R}$ in \cite{Kra01} using the method of Wiener chaos decomposition. The proof works in any dimension and was reproved under an additional smoothness assumption using the method of cumulants, see \cite{Gas21t, Anc24}. In \cite{Kra01}, for the dimension $d = 2$, a first general lower bound for the limiting variance appears, given by the variance of the second chaotic component, which for some constant $C$ is expressed as
\[\Var(Z_\lambda^{(2))}(\phi) = C\int_{\R^2} \|\nabla^2 r(x)\|^2 -2d \langle \nabla r, \nabla r\rangle + d^2r(x)^2\dd x = C\int_{\R^2} (\Delta r(x) + dr(x))^2\dd x,\numberthis\label{eqref:02}\]
where the equality follows from an application of the divergence theorem. The heuristic is valid in any dimension and proves that the limit is positive as long as $r$ is not an eigenfunction of the Laplace operator, which is never the case if $r \in L^2$.
\jump

For vector-valued random fields, i.e., when $k > 1$, few results exist in the literature. The papers \cite{Est16, Nic17, Aza24} prove a CLT for critical points and the Euler characteristic of random fields, and a positive lower bound is given under the assumption that the random field is isotropic or under the existence of a spectral density of the field near the origin. On the additional assumption of a smooth covariance function, normal asymptotics have been proved using the method of cumulants in the recent paper \cite{Anc24}. We first prove a general theorem concerning the asymptotic normality of $Z_\lambda(\phi)$, based on a variance bound inherited from the strategy in \cite{Gas21b} and the standard fourth moment theorem \cite{Nua05} with the contraction principle. We define
\[Z_\lambda^{u,(q+)}(\phi) = \sum_{k=q}^{+\infty} Z_\lambda^{u,(k)}(\phi),\]
the $q$-tail of the variable $Z_\lambda^u(\phi)$.
\begin{theorem}
\label{thm6}
Let $q\geq 1$, and $\Omega\in L^q(V)$. There is a constant $C_q^u$ such that
\[\lim_{\lambda\rightarrow +\infty} \Var(Z_\lambda^{u,(q+)}(\phi)) \rightarrow C_q^u,\]
and
\[Z_\lambda^{u,(q+)}(\phi)-\E[Z_\lambda^{(q+)}(\phi)] \underset{\lambda\rightarrow +\infty}{\longrightarrow} \mathcal{N}(0,C_q^u).\]
\end{theorem}
\subsubsection{First chaos}
Let us study first the projection on the first chaos of $Z_\lambda^u(\phi)$. By symmetry, the first chaotic projection in the case $u=0$ cancels, so the next theorem is relevant for $u\neq 0$. 
\begin{theorem}
\label{thm7}
Assume that $Y(0)\perp\nabla Y(0)$ and that the spectral measure $\psi$ of $Y$ has a continuous density $\Sigma$ w.r.t the Lebesgue measure near the origin. Then there is an explicit positive constant $\alpha_u$ such that
\[\lim_{\lambda\rightarrow+\infty} \Var(Z_\lambda^{u,(1)}(\phi)) = \alpha_u\Sigma(0)(u,u).\]
If additionally, $r$ is integrable and $\Sigma(0)(u,u)>0$ then $Z_\lambda^{u}(\phi)$ satisfies the CLT.
\end{theorem}

One deduce in particular that the first chaos cancels as soon as the spectral measure is supported on a set disjoint from the origin. This is for instance the case for quantities related to random waves, whose spectral measure is supported on the sphere. This general observation corroborates some findings in \cite{Lac22}.\jump

If the spectral measure $\psi$ does not admits a continuous density near the origin, one must use the general Corollary \ref{cor1} to study the exact variance asymptotics of the first chaos. This theorem applies in particular if $\Omega$ is integrable. If the quantity $\Sigma(0)(u,u)$ cancels, one can use ideas similar to the next Theorem \ref{thm2} (stated only for $u=0$) to study the variance of the second chaotic component, see also Remark \ref{rem3}.\jump 

\subsubsection{Second chaos}
The two main obstructions in the case $k>1$ to prove general lower bound for the variance by an analysis of the second chaotic component, is that one cannot assume that the random vector $X(0) = (Y(0), \nabla Y(0))$ follows a standard Gaussian distribution by choosing adequate scalar products on $U$ and $V$, and that it is in general impossible, to compute explicitly the chaotic expansion of the nodal volume without any additional assumption on the field. The computation can be done in particular cases. We define the following set of hypotheses.
\begin{itemize}
\item[(H1)] \underline{There are scalar products on $U$ and $V$ such that $\Var(Y(0),\nabla Y(0)) = \Id$.}\jump
$\quad$ This hypothesis assume a strong independence assumption on the random variable $(Y(0), \nabla Y(0))$, but assume nothing of the independence of processes. Up to changing adequately the scalar product on $U$ and $V$, this hypothesis can be assumed to be true if $k=1$.
\item[(H2)] \underline{There are scalar products on $U$ and $V$ and $v\in U$ such that for all $w\in v^\perp$, $\langle Y,v\rangle\perp \langle Y,w\rangle$}\jump
$\quad$ This hypothesis consider the case of the intersection of the nodal set of a real Gaussian field and the nodal set the nodal set of an another independent Gaussian vector field.
\item[(H3)] \underline{$Y$ is the gradient of a real random field $F$.}\jump
$\quad$ Under this hypothesis, the nodal set of $Y$ is exactly the set of critical points of the real random field $F$. This hypothesis is independent of the first two because of the non-independence of the second derivatives and the possible lack of isotropy.
\end{itemize}

The following theorem proves, under a square integrability property of the covariance function of $Y$, that the variance of the nodal volume converges to a positive constant under one of the following above hypotheses.
\begin{theorem}
\label{thm2}
Assume that $r$ and its second differential are square integrable on $V$, and one of the three following hypotheses (H1), (H2) or (H3) is satisfied. Then there is a finite \textbf{positive} constant $C$ such that
\[\lim_{\lambda\rightarrow +\infty}\Var(Z_\lambda(\phi)) = C,\]
and $Z_\lambda$ satisfies the CLT.
\end{theorem}
This theorem recovers and generalize many previously cited results in the literature concerning the positivity of the lower bound for the nodal volume under the square integrability hypothesis. \jump

We were not able to find a satisfying statement that encompasses the three situations, but the heuristic of proof for the positivity of the limiting variance, given in Section \ref{sec52}, is that in the two cases we are able to compare the density $\Sigma$ of the spectral measure $\mu$ of the process $X = (Y,\nabla Y)$ to a tensor product of two terms: one depending only on the spectral density of the underlying field $\sigma$, and the other one being a polynomial, that comes up from taking the Fourier transform of the derivatives. The conclusion follows by the simple geometric fact that a nonzero square integrable function cannot be supported on the zero set of a non-zero polynomial. In the case where $r\notin L^2(V)$, the situation is more intricate and is the content of the next Section \ref{sec13}.
%
\subsection{Random waves and second chaos cancellation}
\label{sec13}
We use the notation of the previous section. We define $\sigma$ to be the uniform probability measure on the sphere of radius $\sqrt{d}/2\pi$. Let $Y:V\rightarrow U$ be stationary Gaussian random process such that $\Var(Y(0),\nabla Y(0))$ is a standard Gaussian random vector. We say that $Y$ is a \textit{random wave} if its spectral measure $\psi$ is supported on a sphere (necessary of radius $\sqrt{d}/2\pi$). Here we allow the dimension of $U$ to be greater than one, meaning that $Y$ can be seen as a collection of real random waves. Equivalently, $Y$ is a solution of the Laplace equation
\[\Delta Y + d Y=0.\]
We say that the random wave $Y$ is
\begin{itemize}
\item \textit{regular} if its spectral measure admits a continuously differentiable density w.r.t $\sigma$
\item \textit{symmetric} if its covariance function is symmetric
\item \textit{isotropic} if $\psi = \Id\,\sigma$, meaning that $Y$ is a collection of iid real-valued random waves.
\end{itemize}

The isotropic random wave model has a long history. It was informally conjectured by Berry in \cite{Ber77} to be the local limit of high-energy Laplace eigenfunctions on a generic manifold; see the introduction of \cite{Gar23} for the heuristic of the conjecture, relevant literature on the topic, and various possible rigorous formulations of this conjecture. The conjecture implies, in particular, that the nodal set of an eigenfunction tends to equidistribute on the manifold, which supports another long-standing conjecture by Yau \cite{Yau82} concerning the asymptotics of the nodal volume of high-energy eigenfunctions. To obtain a more tractable model, one can approximate an eigenfunction associated with $\lambda$ by a random sum of eigenfunctions with energy located in a small energy window $[\lambda-\varepsilon, \lambda+\varepsilon]$. This model is known as the \textit{monochromatic random wave} model, and it has been shown that it serves as a good approximation to the isotropic Euclidean random wave model; see \cite{Can16, Gas21, Die20}.\jump

On the torus and the sphere, particularly in dimensions $2$ and $3$, where the spectral decomposition of the Laplace operator is explicit, the model of monochromatic random waves and their nodal set has been intensively studied; see \cite{Ros19, Mar23}. In particular, it has been observed that the variance of various local observables of the isotropic random wave (such as the nodal volume \cite{Wig09, Rud08, Kri13, Mar16, Cam19, Dal19, Mar20, Not21, Dal21}, the defect volume \cite{Mar11}, and the Euler characteristic \cite{Cam18b}), in the high-energy limit, is of lower order than one might expect for a "generic" local observable, as suggested by Equation \eqref{eq:02}. This phenomenon is known as the Berry cancellation phenomenon and is purely a mathematical consequence of the second chaos cancellation as $\lambda$ approaches $+\infty$. The proofs presented in the aforementioned papers exploit special properties of Bessel functions and the explicit chaos decomposition of the nodal volume functional to establish such general estimates. In the following, we restrict our study to the Euclidean case, which represents the scaling limit of monochromatic random waves on a generic compact manifold. It is expected that this cancellation phenomenon also appears in a broader Riemannian context; see the forthcoming paper \cite{Ste25}.
\jump

We first prove that this second chaos cancellation can be understood in a simple manner from a spectral point of view by the more general Theorem \ref{thm4}, shedding new light on this cancellation phenomenon, which is valid for a more general class of regular random waves. For instance, the second identity in Equation \eqref{eq:03}, which heuristically explains the cancellation phenomenon from the spatial representation of random waves, becomes trivial in the Fourier domain and reduces to the identity
\[\|\xi\|^4 - 2d\|\xi\|^2 + d^2 = (-\|\xi\|^2+d)^2 = 0\quad\text{if}\quad\|\xi\|=\sqrt{d},\]
which was actually the source of motivation to start the writing of this paper in the first place.
\begin{theorem}
\label{thm3}
Assume that hypothesis (H1) holds true. If $Y$ is not a random wave then 
\[\liminf_{\lambda\rightarrow +\infty}\Var(Z_\lambda^{(2)}(\phi))>0.\]
Conversely, if $Y$ is a regular random wave then
\[\lim_{\lambda\rightarrow +\infty}\Var(Z_\lambda^{(2)}(\phi))=0.\]
\end{theorem}
\begin{theorem}
\label{thm3b}
Assume that hypothesis (H3) holds true, and moreover that the field $F$ has an isotropic distribution. If $F$ is not a random wave then 
\[\liminf_{\lambda\rightarrow +\infty}\Var(Z_\lambda^{(2)}(\phi))>0.\]
Conversely, if $F$ is the isotropic random wave model, then
\[\lim_{\lambda\rightarrow +\infty}\Var(Z_\lambda^{(2)}(\phi))=0.\]
\end{theorem}

In these toy-model examples, one can explicitly compute the variance of the second chaotic projection to provide a partial equivalence between the inclusion of the support of the spectral measure in the sphere and the second chaos cancellation. In more complicated settings, where explicit computation is not always feasible, one can still establish the positivity of the $\liminf$ if the support of the spectral measure is not contained within the zero set of a quadratic form. This approach forms the basis of the proof of Theorem \ref{thm2}, under the additional square-integrability assumption.\jump


The heuristic of these cancellation phenomena, valid for "isotropic" observables such as the nodal volume or critical points, is that the isotropic cone of the second chaotic component must form a sphere, whose radius needs to be determined. In the case of the nodal observables mentioned above, this sphere has a radius precisely $\sqrt{d}/2\pi$. The image of the spectral measure of the field $X = (Y, \nabla Y)$ is a sphere of radius $\sqrt{d}/2\pi$ if and only if $Y$ is a random wave, and the conclusion follows from the second point of Theorem \ref{thm4}.\jump
These theorem are a multidimensional generalization of some of the findings in \cite{Lac21}, where it was proved in the case $V = \R$ that the limiting variance is positive (possibly infinite) if and only if the process is not a cosine with a Gaussian amplitude. In fact, if the spectral measure, restricted to a domain at a positive distance from the sphere of radius $\sqrt{d}/2\pi$, is not square-integrable, then a straightforward byproduct of the proof, along with Lemma \ref{lemma7} and Lemma \ref{lemma8}, shows that the variance diverges to $+\infty$. Together with the discussion after Theorem \ref{thm4}, this situation can also lead to a CLT for the quantity $Z_\lambda$, inherited from a CLT on $Z_\lambda^{(2)}$, provided one can quantify the singularity of the measure $\psi * \psi$ near the origin.\jump

At last, these theorem recover and generalize, for any dimension $d$ and $k$, several of the aforementioned results concerning the second chaos cancellation of the number of critical points of isotropic random waves, and of the volume of the nodal intersection of $k$ isotropic and independent random waves in $\R^d$. For more general $u$-levels, we refer to Remark \ref{rem3}.\jump

To understand the variance asymptotics of $Z_\lambda(\phi)$, one must examine the $4$-th chaos, which is the content of the next Section \ref{sec14}.
\subsection{Asymptotics of the fourth chaos}
\label{sec14}
We showed in the previous paragraph that regular random waves are subject to the general cancellation phenomenon of Theorem \ref{thm4}. In this section, we derive the exact variance asymptotics and a CLT for the nodal volume of regular symmetric random waves, following the heuristic outlined after Theorem \ref{thm4}. The result we obtain encompasses and generalizes many findings in the literature concerning such asymptotics, which have primarily focused on dimensions $d=2,3$ for the isotropic random wave model. We also prove with the same method the exact variance asymptotics and the CLT for the number of critical points of real-valued isotropic random waves, which is a new result in the literature.
\jump

The case $d=2$ for the nodal volume of isotropic random waves is well-studied, with significant results found in works such as \cite{Cam19, Kri13, Nou19} across various settings (spherical harmonics, arithmetic random waves, planar random waves). The exact asymptotics of the variance of the nodal length have been explicitly computed : there exists an explicit positive constant $C$ such that
\[\lim_{\lambda\rightarrow +\infty}\frac{\Var(Z_\lambda(\one_{B(0,1)}))}{\log(\lambda)} = \lim_{\lambda\rightarrow +\infty}\frac{\Var(Z_\lambda^{(4)}(\one_{B(0,1)}))}{\log(\lambda)}= C,\]
and $Z_\lambda$ satisfies a CLT. When $d=3$, it has been proved in \cite{Dal21} the existence of a positive constant $C$ such that
\[\lim_{\lambda\rightarrow +\infty} \Var(Z_\lambda(\one_{B(0,1)}))= C,\]
as well as a CLT. The next theorem proves a general version of these two facts, and includes the case of critical points of a real-valued isotropic random wave.
\begin{theorem}
\label{thm5}
Assume either $Y$ is a symmetric regular random wave, or that $Y$ is the gradient of a real-valued isotropic random wave. Then there is a \textbf{positive} constant $C_d$ such that as $\lambda$ goes to $+\infty$
\[\Var(Z_\lambda(\phi)) \simeq \left\lbrace\begin{array}{llll}
C_2\log(\lambda)\;\text{ if } d=2,\\
C_d\;\qquad\;\;\,\text{ if } d\geq 3,\\
\end{array}
\right.\]
and $Z_\lambda(\phi)$ satisfies the CLT.
\end{theorem}

The heuristic that leads to the distinction between $d=2$ and $d>2$ is as follows, and can already be seen from the general discussion that follows Theorem \ref{thm4}, and corroborates the heuristics for the fluctuations of the polyspectra elaborated in \cite{Gro24}. From the discussion in the previous Section \ref{sec13}, we know that the variance of the second chaotic component cancels, despite having a covariance function that is not square-integrable. To obtain the exact asymptotics of the variance, one must examine the next non-zero chaotic component, namely the fourth one. To compute the fourth chaos and understand the difference between the cases $d=2$ and $d\geq 3$, we need to analyze, in light of Theorem \ref{thm1}, the behavior of the four-fold convolution of the uniform measure on the sphere $\sigma$ near $0$.

\noindent\underline{-- If $d=2$:}\jump

\noindent In this case, one can show that $\sigma * \sigma$ is not square-integrable, or equivalently, that $\Omega \notin L^4(V)$. This means that $\sigma^{*4}$ diverges near the origin. Its exact asymptotics are computed in Lemma \ref{lemma17}, where it is shown to behave as $|\log(x)|$ near the origin. The exact asymptotics of the variance of the fourth chaos then follow from the principles underlying Theorem \ref{thm4}, proving that the limit is non-degenerate. As for the higher chaotic components, note that $\Omega \in L^6(V)$. Their variances converge to a constant, and they do not contribute to the asymptotics of the total variance, which is of order $\log(\lambda)$. In other words, $Z_\lambda$ is asymptotically equal to its fourth chaotic component $Z_\lambda^{(4)}$.\jump

\noindent\underline{-- If $d \geq 3$:}\jump

\noindent In this case, $\sigma * \sigma$ is square-integrable, or equivalently, $\Omega \in L^4(V)$. This means that $\sigma^{*4}(0) = \|\sigma * \sigma\|_2^2$ is well-defined and positive: the variance of the fourth chaos converges to a constant, and we will prove, by the same method used in the case $d=2$, that this constant is positive. The higher chaotic components will also contribute a positive term to the total variance, and unlike the case $d=2$, there is no dominance of a single chaotic component.

\begin{corollary}
Assume that $d\geq 2$ and that the process $Y:V\rightarrow \R$ has an isotropic distribution. Then 
\[\liminf_{\lambda\rightarrow +\infty}\Var(Z_\lambda(\phi))>0.\]
\end{corollary}
This corollary is a straightforward consequence of Theorem \ref{thm3} and Theorem \ref{thm5}. It implies that there cannot be any variance cancellation phenomenon under the assumption of isotropy as soon as $d\geq 2$. In dimension $d=1$ the cosine process with random amplitude is the only counterexample, as asserts \cite{Lac21}.

\section{Hermite polynomials and Wiener chaos expansion}
\label{sec20}
\subsection{Hermite polynomials}
\label{sec21}
In this section, we expose some material about Wiener chaos expansion on a finite dimensional Gaussian space, see \cite{Pec11, Nou12} for general references. The following material is standard but exposed in a way that avoids the use of coordinates to obtain compact and elegant formulas for chaos projections and covariances, in term of the natural scalar product on the space of multilinear maps induced by the underlying Gaussian measure.\jump

Let $W$ be a finite dimensional real vector space of dimension $n$ and $W^*$ be its dual space. For a integer $q\geq 0$ we denote $(W^*)^{\otimes q}$ the space of $q$-linear maps on $W$. Its subspace of symmetric $q$-linear maps can naturally be identified with the space of homogeneous polynomial of degree $q$ on $W$ via the isomorphism $f\mapsto (w\mapsto f(w,\ldots,w))$. In what follows, we will simply write $f(w) = f(w,\ldots,w)$ when $f$ is a symmetric $q$-linear map.\jump

Let $\eta$ be a non degenerate Gaussian measure on $W$. Its variance defines a scalar product on $W^*$, simply denoted by $\langle\,.\,,\,.\,\rangle$ and the associated norm by $\|\,.\,\|$. This pairing between $W$ and $W^*$ allows us to transport this scalar product to $W$, and also to the space of $q$-linear map on $W$, via the formula , for $\ell_1,\ldots,\ell_q,m_1,\ldots,m_q\in W^*$
\[\langle \ell_1\otimes\ldots\otimes \ell_q, m_1\otimes\ldots\otimes  m_q\rangle = \prod_{k=1}^q \langle \ell_k, m_k\rangle.\]
It generalizes the Frobenius scalar product for matrices (i.e. the case $q=2$). We will denote $w^*$ the linear form $\langle w,\,.\,\rangle$.
For a vector $w\in W$, the Radon--Nikodym derivative of the translated measure $\eta(.\, - w)$ with respect to $\eta$ is given by
\[\frac{\dd \eta(.\,-w)}{\dd \eta} :x\longmapsto \exp\left(\langle w,x\rangle-\frac{1}{2}\|w\|^2\right).\]
It allows us to define for $x\in W$ the symmetric $q$-linear Hermite form $H_x^q$ via the Taylor expansion
\[\frac{\dd \eta(.\,-w)}{\dd \eta} : x\mapsto \sum_{q=0}^{+\infty} \frac{1}{q!}H_x^q(w),\qwith H_x^q = (-1)^q\exp\left(\frac{1}{2}\|x\|^2\right) D^q_x\left[ \exp\left(-\frac{1}{2}\|\,.\,\|^2\right)\right]\numberthis\label{eq:06}.\]
Let $U$ be another finite dimensional vector space equipped with another non-degenerate Gaussian measure $\tau$. The induced scalar product will also be denoted by $\langle\,.\,,\,.\,\rangle$. We define the set
\[E(\eta,\tau) = \enstq{\Omega\in(W^*\otimes U^*)^*}{\forall \ell\in W^*,\;\;m\in U^*,\quad \|\ell\|^2 + \|m\|^2 + 2\Omega(\ell,m)\geq 0},\]
and $E^*(\eta,\tau)$ the interior of $E(\eta,\tau)$, i.e. where the inequality in the definition is strict. When $W=U$ and $\eta=\tau$ we simply denote them $E(\eta)$ and $E^*(\eta)$. For $\Omega\in E(\eta,\tau)$, we define the Gaussian measure $\eta\otimes_{\Omega}\tau$ on $W\times U$ so that its marginals are given by $\eta$ and $\tau$ respectively, and such that $\Omega$ is the covariance bilinear map between the two marginals. 
For instance, the choice of $\Omega = 0$ leads to the product measure on $W\times U$.  One has following Gaussian integral, for $w\in W$ and $u\in U$
\[\int_{W\times U} \exp\left(\langle w,x\rangle - \frac{1}{2}\|w\|^2\right)\left( \langle u,y\rangle - \frac{1}{2}\|u\|^2\right)\dd\eta\!\otimes_{\Omega}\!\tau(x,y) = \exp\left(\Omega(w^*,u^*)\right).\numberthis\label{eq:07}\]
Developing inside the integral on the left using Equation \eqref{eq:06},  and using the Taylor expansion of the exponential of the right, one get by identifying the $q$-linear maps for different $q$, the relations for $q,q'\geq 0$
\[\int_{W\times U}H_x^q(w)H_y^{q'}(u)\dd\eta\otimes_{\Omega}\tau(x,y) = \delta_{qq'}q!\Omega(w,u)^q = \delta_{qq'}q!\Omega^{\otimes{q}}((w^*)^{\otimes q},(u^*)^{\otimes q})),\]
where $\Omega^{\otimes q}$ is seen as a bilinear map on $(W^*)^{\otimes q}\times (U^*)^{\otimes q}$. 
\subsection{Wiener chaos expansion}
\label{sec22}
From now on, we assume that $W=U$ and $\eta =\tau$, so that $\Omega$ is a bilinear map on $W^*$. The choice of $\Omega = \langle \,.\,,\,.\,\rangle$ in the previous paragraph leads to $\eta\otimes_\Omega\eta$ being the diagonal measure on $W\times W$ induced by $\eta$. In that case, Equation \eqref{eq:07} yields the classical orthogonality relations
\[\int_W H_x^q(w)H_x^{q'}(u)\dd \eta(x) = \delta_{qq'}q!\langle w,u\rangle^q.\numberthis\label{eq:08}\]
This relation allows us to detail the classical Wiener chaos decomposition for the space $L^2(\eta)$ in our framework. We define the $q$-Hermite chaos spaces of functions 
\[\HH_q = \Vect\left(\left(x\mapsto H^q_x(w)\right)_{w\in W}\right).\] The previous orthogonality relation \eqref{eq:08} implies that the vector spaces $\HH_q$ and $\HH_{q'}$ are orthogonal for different $q,q'$. Moreover, their union coincides with the set of polynomials on $W$ and is thus dense in $L^2(\eta)$. We directly deduce the orthogonal decomposition 
\[L^2(\eta) = \overline{\bigoplus_{q=0}^{\infty} \HH_q}.\]
For $f\in L^2(\eta)$, we define the symmetric $q$-linear map $f_q$ on $W$ and the $q$th Hermite projection $\pi_q(f)$ as
\[f_q = \int_{W} f(x)H_x^q\dd\eta(x)\quand \pi_q(f):x\mapsto \frac{1}{q!}\langle f_q,H^q_x\rangle.\]
Then for $q,q'\in\N$ and $w\in W$, one has
\[\pi_q(H_{.}^{q'}(w))(y) = \frac{1}{q!}\left\langle \int_W H_x^{q'}(w)H_x^q\dd\eta(x)\,,\,H_y^q\right\rangle  = \delta_{qq'}\left\langle(w^*)^{\otimes q}, H_y^q\right\rangle = \delta_{qq'}H_y^q(w),\]
where the second equality follows from Equation \eqref{eq:08}. It shows that the operator $\pi_q$ is indeed the orthogonal projection of $L^2(\eta)$ onto the $q$-chaos space $\HH_q$. If we go back to a general coupling $\Omega$ between $\eta$ and itself on $W^*\times W^*$, one obtain, for $f,g\in L^2(\eta)$, the following decomposition
\begin{align*}
\int_{W\times W} f(x)g(y)\dd\eta\otimes_{\Omega}\eta(x,y) &= \sum_{q,q'=0}^{+\infty}\int_{W\times W} \pi_qf(x)\pi_{q'} g(y)\dd\eta\otimes_{\Omega}\eta(x,y)\\
&=\sum_{qq'=0}^{+\infty}\frac{1}{q!q'!}\int_{W\times W}\langle f_q,H_x^q\rangle\langle g^{q'},H_x^{q'}\rangle \dd \eta\otimes_{\Omega}\eta(x,y) \\
&=\sum_{q=0}^{+\infty}\frac{1}{q!}\Omega^{\otimes q}(f_q,g_q).\numberthis\label{eq:13}
\end{align*}
\subsection{Hermite expansion of distributions}
\label{sec23}
Let $f$ be a tempered distribution on $W$, and $(f_q)_q$ as sequence of $q$-linear forms such that it holds in the sense of tempered distribution the equality
\[f = \sum_{q=0}^{+\infty} \langle f_q, H_x^q\rangle.\]
Note by a density argument, one necessary has
\[f_q = f(h_q),\qwith h_q:x\mapsto H_x^q \exp\left(-\frac{\|x\|^2}{2}\right).\]
For $\Omega\in  E^*(\eta)$, the density $\rho_\Omega$ of the measure $\eta\otimes_\Omega\eta$ w.r.t the Lebesgue measure on $W\times W$ is well-defined and is a function in the Schwartz class. Then the mapping $C_f : E^*(\eta)\rightarrow \R$ given by
\[C_f:\Omega\mapsto (f\otimes f)(\rho_\Omega)\]
is well defined, and as before it holds that for $\Omega\in E^*(\eta)$,
\[C_f(\Omega) = \sum_{q=0}^{+\infty}\frac{1}{q!}\Omega^{\otimes q}(f_q,f_q).\]
It shows that the function $C_f$ is analytic on $E^*(\eta)$, and the right-hand side corresponds to its series expansion. Note that if $f\in L^2(\eta)$, the function $C_f$ can be continuously extended to $E(\eta)$. Let $A$ be a compact set of $E^*(\eta)$ and $\Omega\in A$. The integral remainder in Taylor expansion implies that
\[\left|C_f(\Omega) - \sum_{k=0}^{q-1}\Omega^{\otimes k}(f_k,f_k)\right|\leq \left(\sup_{\Omega\in A} \frac{1}{q!}\|D^q C_f\|\right)\|\Omega\|^q.\numberthis\label{eq:14}\]
This expression is an alternative to Arcones inequality, see \cite{Arc94}, which has the advantage of being valid for any tempered distribution $f$. This bound is inspired from the results in \cite{Gas21b}, and holds without any square integrability assumption on $f$, the downside being that one must restrict $\Omega$ to $E^*(\eta)$, otherwise $C_f$ might be ill-defined (a typical example is the choice of function $f = \delta_0$).

\section{Positive definite functions, Hermitian form and convolution}
\label{sec30}
The Fourier transform of a covariance function of a process $Y:V\rightarrow W$ is a measure taking values in the space of Hermitian forms on $W^*$. In the following, we expose the basic theory of such measure.
\subsection{Positive Hermitian forms}
\label{sec31}
Let $W$ be a finite dimensional complex vector space. We denote by $\overline{W}$ the conjugate vector space of $W$ and $\Sesq(W) = W^*\otimes \overline{W}^*$, the space of sesquilinear forms on $W^*$, which is naturally identified with the space of linear maps between $\overline{W}$ and $W^*$. The \textit{image} of an element $\Sigma\in \Sesq(W)$, denoted by $\Imm(\Sigma)$, is the image of this mapping, as a subspace of $W^*$. If $\Sigma\in \Sesq(W)$, we define $\Sigma^c\in \Sesq(W)$, the complex conjugate of $\Sigma$, via the formula for $w_1,w_2\in W$
\[\Sigma^*(w_1,w_2) = \overline{\Sigma(w_2,w_1)}.\]
Let $\HH(W)$ be the space of Hermitian forms on $W$, that is the subspace of sesquilinear forms on $W$ invariant by complex conjugate. We simply write $\Sigma(w) = \Sigma(w,w)$ when $\Sigma$ is a Hermitian form, since it depends only on the diagonal evaluations via the polarization identity.
\[\Sigma(w_1,w_2) = \frac{1}{4}\left(\Sigma(w_1+w_2)-\Sigma(w_1-w_2)+i\Sigma(w_1+iw_2)-i\Sigma(w_1-iw_2)\right).\numberthis\label{eq:09}\]
An Hermitian form $\Sigma$ is positive semi-definite if additionally, for all $w\in W$, $\Sigma(w)\geq 0$. We will denote by $\HH^+(W)$ the set of positive semi-definite Hermitian form on $W$.\jump

We now prove a few technical lemmas concerning Hermitian forms. Let $\Sigma_1,\ldots,\Sigma_q$ Hermitian forms on complex vector spaces $W_1,\ldots,W_q$. We define the Hermitian form on $W_1\otimes \ldots\otimes W_q$
\[\bigotimes_{k=1}^q \Sigma_k : (w_1\otimes\ldots\otimes w_q)\mapsto \prod_{k=1}^q \Sigma_k(w_k).\]
In particular, if $\Sigma$ is a Hermitian form on $W$, then $\Sigma^{\otimes q}$ is a Hermitian form on $W^{\otimes q}$.
\begin{lemma}
\label{lemma1}
Let $\Sigma_1,\ldots,\Sigma_q$ positive semi-definite Hermitian forms on complex vector spaces $W_1,\ldots,W_q$. Then $\bigotimes_{k=1}^q \Sigma_k$ is a positive semi-definite Hermitian form on $W_1\otimes \ldots\otimes W_q$.
\end{lemma}
\begin{proof}
By the diagonalization theorem for Hermitian forms, for $1\leq k\leq q$ there is a basis of $W_k$, denoted $w^1_k,\ldots,w^{d_k}_k$ such that $\Sigma_k(w^i_k,w^j_k) = 0$ if $i\neq j$. Let $w\in\bigotimes_{k=1}^qW_k$. Decomposing $w$ on the induced basis on $\bigotimes_{k=1}^qW_k$, so that
\[w = \sum_{i_1,\ldots,i_q=1}^{d_{i_1},\ldots,d_{i_q}} \alpha_k^{i_k}\bigotimes_{k=1}^qw_k^{i_k}\]
 one gets
\[\bigotimes_{k=1}^q \Sigma_k(w) = \sum_{i_1,\ldots,i_q=1}^{d_{i_1},\ldots,d_{i_q}}\prod_{k=1}^q|\alpha_k^{i_k}|^2\Sigma_k(w_k^{i_k})\geq 0,\]
which proves that the measure $\mu_1\otimes\ldots\otimes\mu_q$ is positive.
\end{proof}
\begin{lemma}
\label{lemma2}
Let $\Sigma\in \HH^+(W^*)$, $f\in W^*\otimes W^*$ and $\langle\,\,.\,,\,.\,\rangle$ a Hermitian product on $W$. Assume the existence of a subspace $V$ of $W$, and application $P : W^*\rightarrow V$ so that $\Sigma(\cdot,\cdot) = \langle P(\cdot), P(\cdot)\rangle$. Then
\[\Sigma^{\otimes 2}(f,f) = \|f(P(\cdot),P(\cdot))\|^2.\]
\end{lemma}
\begin{proof}
Fixing a basis on $W$ this identity simply reads
\[\Tr(\bar{P}P^TfP\bar{P}^T\bar{f}) = \|P^TfP\|^2.\]
\end{proof}
Given a symmetric element $f\in W^*\otimes W^*$, we define its \textit{isotropic cone} $C(f)$ as the set
\[C(f) = \enstq{w\in W}{f(w)=0}.\]
\vspace{-15pt}
\begin{lemma}
\label{lemma3}
Let $\Sigma\in \HH^+(W^*)$ and $f\in W\otimes W$. Then 
\[\Sigma^{\otimes 2}(f,f)=0\Longleftrightarrow \Imm(\Sigma)\subset C(f).\]
\end{lemma}
\begin{proof}
Let $\langle\,\,.\,,\,.\,\rangle$ be an Hermitian product on $W$. Since  $\Sigma$ is a non-negative Hermitian form, it admits a unique Hermitian square root $\sqrt{\Sigma}:W^*\rightarrow W$, so that
\[\Sigma(\cdot,\cdot) = \langle \sqrt{\Sigma}(\cdot),\sqrt{\Sigma}(\cdot)\rangle,\]
From the previous lemma,
\[\Sigma^{\otimes 2}(f,f) = \left\|f\left(\sqrt{\Sigma}(\,.\,),\sqrt{\Sigma}(\,.\,)\right)\right\|^2.\]
Since the maps $\Sigma$ and $\sqrt{\Sigma}$ have the same image, the right hand side cancels if and only if $\Imm(\Sigma)\subset C(f)$.
\end{proof}

\subsection{Positive Hermitian measure and Bochner Theorem}
\label{sec32}
Let $V$ be a finite dimensional real vector space and $\BB(V)$ be its Borel $\sigma$-field. Let $W$ be a finite dimensional complex vector space. A map $\mu : \mathcal{B}(V)\rightarrow  \Sesq(W)$ is a \textit{Hermitian measure} over $W$ if for all $w\in W$, the mapping $\mu_w:E\mapsto \mu(E)(w)$ is a bounded signed measure on $(V,\mathcal{B}(V))$, and is \textit{positive} if it is a non-negative measure. If one fix a basis of $W$, then on can think of $\mu$ as a matrix of complex-valued measures such that for all measurable set $E$, such that the quantity $\mu(E)$ is a Hermitian matrix, positive semi-definite if $\mu$ is positive. \jump

A function $\Omega:V\rightarrow \Sesq(W)$ is positive semi-definite if for all choice of vectors $v_1,\ldots,v_m\in V$ and $w_1,\ldots,w_m\in W$, one has
\[\sum_{i,j=1}^m \Omega(v_i-v_j)(w_i,w_j)\in\R_+.\]
When $W=\C$, this is the usual definition of positive semi-definite functions. Bochner theorem then states that there is an equivalence between being a continuous positive semi-definite complex valued function, and being the Fourier transform of a bounded measure on the real line. We show that this theorem has a straightforward extension to functions taking values in $\Ss q(W)$.
\begin{lemma}
\label{lemma4}
Let $\Omega:V\rightarrow \Sesq(W)$. The two propositions are equivalents.
\begin{itemize}
\item[(1)] $\Omega$ is a continuous positive semi-definite function.
\item[(2)] $\Omega$ is the Fourier transform of a positive Hermitian measure $\mu$ on $V^*$ over $W$.
\end{itemize}
\end{lemma}
\begin{proof}
Let us prove first the direct implication. For $w\in W$, the function $\Omega_w:v\mapsto \Omega(v)(w,w)$ is a continuous complex valued positive semi-definite function. By the usual Bochner Theorem, $\Omega_w$ is the Fourier transform of a non-negative bounded measure on $V^*$, denoted $\mu_w$. Let $E$ be a measurable set in $V^*$. The linearity of the Fourier transform and the polarization identity \eqref{eq:09} imply that the mapping 
\[\mu(E):w\mapsto \mu_w(E)\]
defines a positive semi-definite Hermitian form on $W$ and that $\hat{\mu} = \Omega$. Conversely, Assume that $\Omega=\hat{\mu}$ is a positive Hermitian measure on $V^*$ over $W$. Since the measure $\mu$ is bounded, the function $\Omega$ is continuous. Let $v_1,\ldots,v_m\in V$ and $w_1,\ldots,w_m\in W$. Then
\[0\leq \int_{V^*}\dd\mu(\xi)\left(\sum_{i=1}^m w_ie^{2i\pi\xi(v_i)}\right) = \sum_{i,j=1}^m\int_{V^*}e^{2i\pi\xi(x_i-x_j)}\dd\mu(w_i,w_j) = \sum_{i,j=1}^m \Omega(x_i-x_j)(w_i,w_j).\]
This shows that $\Omega$ is a positive semi-definite function.
%
\end{proof}
\begin{remark}
If $W$ is a real vector space instead of a complex vector space, we can embed it in its complexification $W_\C = W\otimes\C$ via the mapping $w\mapsto w\otimes 1$. A map $\Omega:V\rightarrow (W^*)^{\otimes 2}$ can be seen as a map taking values in $\Sesq(W_\C)$ via this embedding, and one can extend the notion of positive definiteness for such bilinear maps $\Omega$. The conclusion of Lemma \ref{lemma4} then holds if one replace $(2)$ by saying that $\Omega$ is the Fourier transform of a positive Hermitian measure over $W_\C$. In particular, if $X:V\rightarrow W$ is a stationary Gaussian process with continuous covariance function $\Omega$, then its spectral measure $\mu = \widehat{\Omega}$ is a symmetric positive Hermitian measure over $W_\C$.
\end{remark}
\subsection{Convolution of Hermitian measures}
\label{sec33}
Let $q\geq 1$, $W_1,\ldots,W_q$ be complex vector spaces and $\mu_1,\ldots,\mu_q$ be Hermitian measures over $W_1,\ldots,W_q$ respectively. We define the Hermitian measure $\mu_1*\ldots*\mu_q$ on $V$ over $\bigotimes_{k=1}^qW_k$ by the formula, for all vectors $w_1,\ldots, w_q$ in $W_1,\ldots,W_q$ respectively, by the formula
\[\mu_1*\ldots*\mu_q(w_1\otimes \ldots\otimes w_q) = \mu_1(w_1)*\ldots*\mu_q(w_q),\]
where on the right, the symbol $*$ denotes the usual convolution of measures. Alternatively, one can construct the measure $\mu_1*\ldots*\mu_q$ as the push-forward of the Hermitian measure $\mu_1\otimes\ldots\otimes \mu_q$ on $V^q$ by the addition map from $V^q$ to $V$, meaning that for a measurable subset $E$ of $V$ one has
\[\mu_1*\ldots*\mu_q(E) = \int_{V^q} \one_E(x_1+\ldots+x_q)\dd (\mu_1\otimes\ldots\otimes \mu_q)(x_1,\ldots,x_q).\]
 If $\mu$ is a measure on $V$ taking value in the space of Hermitian forms on a complex space $W$, we define its $q$-fold convolution $\mu^{*q}$ as the convolution with itself $q$ times.
\begin{lemma}
\label{lemma5}
Assume that for $1\leq k\leq q$, the measures $\mu_k$ are positives Hermitian measures over $W_1,\ldots,W_q$ respectively. Then the measures $\mu_1\otimes\ldots\otimes \mu_q$ and $\mu_1*\ldots*\mu_q$ are positive Hermitian measures.
\end{lemma}
\begin{proof}
The positivity of $\mu_1\otimes\ldots\otimes\mu_q$ follows directly from Lemma \ref{lemma1}. The positivity of $\mu_1*\ldots*\mu_q$ then follows from the definition of convolution.
\end{proof}

Let $\nu$ be a bounded measure on $V$. We say that a positive Hermitian measure $\mu$ is \textit{absolutely continuous} w.r.t $\nu$ if for all $w\in W$, the measure $\mu_w$ is absolutely continuous w.r.t to $\nu$, and we write it $\mu \ll \nu$. For instance, if $W$ comes equipped with an Hermitian scalar product (so one can take the trace of a sesquilinear form on $W$), then one can define the trace measure $\mu^{\Tr} : E\mapsto \Tr(\mu(E))$, and it is straightforward to see that $\mu\ll\mu^{\Tr}$. By the Radon-Nikodym theorem, there exists a non-negative function $\Sigma:V\rightarrow \Ss q(W)$, uniquely defined up to a $\nu$-null set, such that $\dd \mu = \Sigma\,\dd \nu$. Such a decomposition of $\mu$ will be called a \textit{representation of $\mu$}. Note that for such a representation, one always has $\mu^{\Tr}\ll\nu$.\jump

Given a representation of $\mu$, we define the \textit{image} of $\mu$, denoted by $\Imm(\mu)$, as the smallest closed subset $A\subset W^*$ such that
\[\nu\left(\enstq{x\in V}{\Imm(\Sigma(x))\not\subset A}\right)=0.\]
This notion coincides with the essential range of the measurable function $(x,w)\mapsto \Sigma(x)(w,.)$, and is independent of the representation of $\mu$. 
\begin{lemma}
\label{lemma6}
The image of a positive Hermitian measure $\mu$ does not depend on a choice of representation.
\end{lemma}

\begin{proof}
Let $\rho$ be another measure such that $\mu\ll \rho$. Up to replacing $\rho$ by $\nu+\rho$, and symmetry, we can assume that $\nu\ll\rho$. Let $\dd\mu = \Gamma \dd\rho$ and $\dd \nu = \alpha\,\dd\rho$. By the chain rule, $\Gamma = \alpha\Sigma$, so that $\Gamma$ and $\Sigma$ have the same images on the subset $\{\alpha>0\}$. The conclusion follows, since $0\in A$ and the measures $\nu$ and $\rho$ share the same null sets on the subset $\{\alpha>0\}$.
\end{proof}

\subsection{Hermitian measure and approximation of unity}
\label{sec34}
Let $\gamma = |\hat{\phi}|^2$, where $\phi$ is a test function as defined in introduction, and recall, for a measure $\nu$, the definitions of $D^-\nu$, $D^+\nu$ and $D\nu$. The proof of the following lemma is inspired from \cite{Sae96}.

\begin{lemma}
\label{lemma7}
Let $\nu$ be a bounded measure on $V$. Then
\[\liminf_{\lambda\rightarrow +\infty} \int_{V}\gamma_\lambda(x)\dd\nu(x)\geq  \|\gamma^-\|_1D\nu^{-}(0),\]
and
\[\limsup_{\lambda\rightarrow +\infty} \int_{V}\gamma_\lambda(x)\dd\nu(x)\leq \|\gamma^+\|_1 D\nu^{+}(0).\]
\end{lemma}
\begin{proof}
One can assume that $\gamma$ is radially decreasing, since one has the bounds $\gamma^-\leq \gamma\leq \gamma^+$. One has
\[\int_V \gamma_\lambda(x)\dd x = \int_0^{+\infty} \Vol\{\gamma_\lambda>t\}\dd t,\numberthis\label{eq:10}\]
meaning that the function $t\rightarrow \Vol\{\gamma_\lambda>t\}$ is integrable on $\R_+$. More generally
\begin{align*}
\int_V \gamma_\lambda(x)\dd \nu(x)&=\int_V \int_{0}^{\gamma_\lambda(x)}\dd t\dd \nu(x)\\
&= \int_0^{+\infty} \nu\{\gamma_\lambda>t\}\dd t\\
&=\int_0^{+\infty} \frac{\nu(\frac{1}{\lambda}\{\gamma>t\})}{\lambda^d}\dd t\\
&= \int_0^{+\infty} \frac{\nu(\frac{1}{\lambda}\{\gamma>t\})}{\Vol(\frac{1}{\lambda}\{\gamma>t\})}\Vol\{\gamma>t\}\dd t
\end{align*}
The set $\{\gamma>t\}$ is a centered ball, whose radius is decreasing and finite as soon as $t>0$. For $\varepsilon>0$, there is a constant $C_\varepsilon$ such that
\[\int_V \gamma_\lambda(x)\dd \nu(x) \geq \inf_{R\leq \frac{C_\varepsilon}{\lambda}}\frac{\nu(B(0,R))}{\Vol(B(0,R))}\int_{\varepsilon}^{+\infty}\Vol\{\gamma>t\}\dd t.\]
Taking the $\liminf$ as $\lambda$ goes to $+\infty$, then the limit as $\varepsilon$ goes to zero, one deduce the first statement thanks to Equation \eqref{eq:10}. As for the second statement, it is obvious if $D\nu^+(0)=+\infty$, so one can assume that it is finite. The integrand in the last line is non-negative and bounded above by the quantity $M\nu(0)\Vol\{\gamma>t\}$, which is integrable from Equation \eqref{eq:10} and the conclusion follows by dominated convergence.
\end{proof}
For a non-negative bounded measure $\nu$ on $V$, we define $\|\nu\|_2 = \|f\|_2$ if the measure $\nu$ has a density $f\in L^2(V)$ w.r.t the Lebesgue measure, and $\|\nu\|_2 = +\infty$ otherwise. Alternatively, by Plancherel theorem, $\|\nu\|_2 = \|\hat{\nu}\|_2$, where $\hat{\nu}$ is the Fourier transform of $\nu$, which is a continuous bounded function on $V^*$.
\begin{lemma}
\label{lemma8}
Let $\nu$ be a non-zero symmetric bounded measure on $V$. Then
\[D\nu*\nu (0) = \|\nu\|_2^2.\]
\end{lemma}
\begin{proof}
Let $h$ be the Fourier transform of the function $\frac{\one_{B(0,1)}}{\Vol(B(0,1))}$. Then $|h(x)|\leq h(0)=1$, $h$ is continuous, and
\[\frac{\nu*\nu(B(0,R))}{\Vol(B(0,R))} = \int_{V^*}h\left(\frac{x}{R}\right)\hat{\nu}^2(x)\dd x.\]
If $\|\nu\|_2<+\infty$ then this quantity converges towards $\|\nu^2\|_2^2$ by dominated convergence and Plancherel theorem. If $\|\nu\|_2=+\infty$, one has the lower bound
\[\frac{\one_{B(0,R)}}{\Vol(B(0,R))}\geq \frac{1}{2^d}\frac{\one_{B(0,R/2)}}{\Vol(B(0,R/2))}*\frac{\one_{B(0,R/2)}}{\Vol(B(0,R/2))}.\]
One deduce by Plancherel
\[\frac{\nu*\nu(B(0,R))}{\Vol(B(0,R))}\geq \frac{1}{2^d}\int_{V^*}h^2\left(\frac{x}{2R}\right)^2\hat{\nu}^2(x)\dd x\geq \int_{B(0,A)}h^2\left(\frac{x}{2R}\right)^2\hat{\nu}^2(x)\dd x,\]
for $A>0$. Given such $A$, there is $R$ large enough such that the integrand is an increasing function of $R$ towards the function $\hat{\nu}^2$, and so by monotone convergence,
\[\liminf_{R\rightarrow 0} \frac{\nu*\nu(B(0,R))}{\Vol(B(0,R))} \geq \int_{B(0,A)}\hat{\nu}^2(x)\dd x.\]
The conclusion follows by letting $A$ goes to $+\infty$.
\end{proof}

\begin{lemma}
\label{lemma9}
Let $E$ be a finite dimensional vector space. Let $\nu$ be a bounded measure on $V$, $\Sigma:V\rightarrow E$ be a measurable function and $f:E^2\rightarrow \R$ be a continuous function. Assume that 
\[\nu\left(\enstq{x\in V}{f(\Sigma(x),\Sigma(x))>0}\right)>0.\]
Then there is $\alpha>0$ and a measurable subset $C$ of $V$ with positive measure such that
\[\forall x,y\in C,\quad f(\Sigma(x),\Sigma(y))\geq \alpha.\]
\end{lemma}
\begin{proof}
Let $F:(x,y)\mapsto f(\Sigma(x),\Sigma(y))$. We define
\[A = \enstq{x\in V}{F(x,x)\geq 2\alpha}.\]
There are constants $\alpha,\varepsilon>0$ such that $\nu(A)>2\varepsilon$. By Lusin's theorem, there is a compact subset $K\subset V$ such that $\nu(A\setminus K)\leq \varepsilon$ and such the restriction of $\Sigma$ to $A\cap K$ is uniformly continuous. Let $B=A\cap K$. Then $\nu(B)\geq \varepsilon$ and the restriction of $\Sigma$ to $B$ is a uniformly continuous function. The function $F$ is thus uniformly continuous on $B\times B$. We define $\omega$ its the modulus of continuity, meaning in particular that for $x,y\in B$ with $\|x-y\|\leq \eta$,
\[F(x,y) \geq 2\alpha - \omega(\eta).\]
Let us choose $\eta$ small enough so that $\omega(\eta)\geq \alpha$. The choice of a measurable subset $C\subset B$ of positive measure with diameter lower than $\eta$ concludes the proof.
\end{proof}
\begin{lemma}
\label{lemma10}
let $\mu$ be a Hermitian measure on $V$ taking values in $\Ss q^+(W^*)$, and $f$ be a symmetric element of $W^*\otimes W^*$. If $\Imm(\mu)\not\subset C(f)$ then
\[D^-\mu^{*2}(f)(0)>0.\]
\end{lemma}
\begin{proof}
Assume that $\mu$ has the representation $\mu = \Sigma\dd\nu$, and define the function
\[F(x,y) = \left(\Sigma(x)\otimes \Sigma(y)\right)(f).\]
Since $\Imm(\nu)\not\subset C(f)$, then thanks to the equivalence given by Lemma \ref{lemma3}, one has
\[\nu\left(\enstq{x\in V}{F(x,x)>0}\right)>0.\]
By the previous Lemma \ref{lemma9}, there is $\alpha > 0$ and a subset $C$ of $V$ such that $\nu(C)>0$ and such that for $x,y\in C$, $F(x,y)\geq \alpha$. It implies that for $R>0$,
\begin{align*}
\frac{\mu^{*2}(f)(B(0,R))}{\Vol(B(0,R))} &= \frac{1}{\Vol(B(0,R))}\int_{\|x-y\|\leq R}F(x,y)\dd\nu(x)\dd\nu(y)\\
&\geq \frac{\alpha}{\Vol(B(0,R))}\int_{\|x-y\|\leq R}\one_C(x)\one_C(y)\dd\nu(x)\dd\nu(y)\\
&\geq \alpha \frac{\nu_C*\nu_C(B(0,R))}{\Vol(B(0,R))},
\end{align*}
where $\nu_C$ is the measure on $V$ defined by the restriction of $\nu$ to $C$. Since $\nu(C)>0$, $\|\nu_C\|_2>0$ and the conclusion follows immediately from Lemma \ref{lemma8}.
\end{proof}
\begin{lemma}
\label{lemma11}
let $\mu$ be a Hermitian measure on $V$ taking values in $\Ss q^+(W^*)$, and $f$ a Hermitian form on $W$ such that $\Imm(\mu)\subset C(f)$. Assume that $\mu$ has the representation $\mu = \Sigma\dd \nu$, where either
\begin{itemize}
\item[(1)] $\nu$ is the Lebesgue measure of $V$ and $\Sigma\in L^2(V)$
\item[(2)] $\nu$ is the Lebesgue measure of a smooth compact hypersurface $M$ of $V$ and $\Sigma \in C^1(M)$.
\end{itemize}
Then
\[D\mu^{*2}(f)(0)=0.\]
\end{lemma}
\begin{proof}
Given a Riemannian submanifold $M$ of $V$ of codimension $s$, let $\dist$ be its induced Riemannian distance, $\Vol_M$ is volume measure, and for $x\in M$ and $R>0$, let $B(x,M)$ the geodesic ball centered on $x$. If $\Sigma$ is a function defined on $M$, we define
\[\Sigma_R :x\mapsto \frac{1}{\Vol_M(B_M(x,R))}\int_{B_M(x,R)}\Sigma(y)\dd \nu(y).\]
If $M = V$ and $\Sigma\in L^2(V)$, then for all $R>0$, one also has that $\Sigma_R\in L^2(V)$ and by a standard approximation technique, $\|\Sigma_R-\Sigma\|_2 = o(1)$ as $R$ goes to $0$. Similarly, if $M$ is a compact hypersurface and $\Sigma\in C^2(M)$, then using the Taylor expansion $\Sigma(y) = \Sigma(x) + \nabla_x\Sigma\cdot(y-x) + o(\|y-x\|)$, one has by symmetry
\[\|\Sigma_R - \Sigma\|_\infty = o(R).\]
In cases $(1)$ and $(2)$ we get the estimate for small $R$
\[\|\Sigma_R-\Sigma\|_2 = o(R^s).\numberthis\label{eq:11}\]
Note that for 
$s\geq 2$ the Taylor expansion of $\Sigma_R$ yields a non-zero second order term, which explained why we restricted ourselves to the case $s\geq 1$. Since $\Imm(\mu)\subset C(f)$ then for almost all $x\in M$,
\[\left(\Sigma(x)\otimes \Sigma(x)\right)(f) = 0.\]
There is a positive constant $C$ such that for $R$ small enough and $x,y\in M$ with $\dist(x,y)\leq M$,
\[\frac{\Vol_M(B_M(x,R))}{\Vol(B(0,R/2))}\leq \frac{C}{R^s},\quand \dist(x,y)\leq 2\|x-y\|.\]
One then has
\begin{align*}
\frac{\mu^{*2}(f)(B(0,R/2))}{\Vol(B(0,R/2))} &=\frac{1}{\Vol(B(0,R/2))}\int_{\|x-y\|\leq \frac{R}{2}}\left(\Sigma(x)\otimes \Sigma(y)\right)(f)\dd\nu(x)\dd\nu(y)\\
&\leq \frac{1}{\Vol(B(0,R))}\int_{\dist(x,y)\leq R}\left(\Sigma(x)\otimes \Sigma(y)\right)(f)\dd\nu(x)\dd\nu(y)\\
&\leq \frac{C}{R^{d-s}}\left(\int_M \left(\Sigma(x)\otimes \Sigma_R(x)\right)(f)\dd\nu(x) - \int_M \left(\Sigma(x)\otimes \Sigma(x)\right)(f)\dd\nu(x)\right)\\
&\leq \frac{C}{R^s}\|f\|^2 \int_M \|\Sigma(x)\|\|\Sigma_R(x)-\Sigma(x)\|\dd\nu(x)\\
&\leq \frac{C}{R^s}\|f\|^2\|\Sigma\|_{L^2(\nu)}\|\Sigma-\Sigma_R\|_{L^2(\nu)}\\
&\leq o(1),
\end{align*}
where the last line follows from Equation \eqref{eq:11}. The conclusion follows by passing to the limit as $R$ goes to $0$.
\end{proof}
\newpage
\section{Kac--Rice formula, chaotic expansion and random waves}
\label{sec40}
In this section, we first expose standard material concerning the Kac--Rice formula, in particular proving that one can formally embed its studies in light of the general framework of Section \ref{sec11}. We then compute the first terms of the chaotic expansion of the nodal volume. In a second part, we study the regularity of the convolution of spherical measures, which will be crucial for proving the fourth chaos asymptotics for the nodal volume of random waves.
\subsection{The nodal measure of random fields}
\label{sec41}
In this section we use the notation of the introduction. Let $(\delta_u^\varepsilon)_{\varepsilon > 0}$ be an approximation of the Dirac mass at $0$ in $U$, denoted by $\delta$. We then define for $\phi$ a test function and $\lambda>0$ the approximated quantity
\[Z_\lambda^{u,\varepsilon}(\phi) = \int_U\delta_u^\varepsilon(x)Z_\lambda^x(\phi)\dd x.\]
We define the gradient field $\nabla Y : V\rightarrow V\otimes U$. By the co-area formula, one has the a.s. equality
\[Z_\lambda^{u,\varepsilon}(\phi) = \frac{1}{\lambda^{d/2}}\int_V \phi\left(\frac{v}{\lambda}\right)\delta_u^\varepsilon(Y(v))J(\nabla_vY)\dd v,\]
where for an element $D\in U\otimes V$, $J(D) = \sqrt{\det (DD^T)}$.
\begin{lemma}
\label{lemma12}
We have the following convergence
\[Z_\lambda^{u,\varepsilon}(\phi)\underset{\varepsilon\rightarrow 0}{\longrightarrow} Z_\lambda^u(\phi)\]
a.s. and in $L^2(\Omega)$.
\end{lemma}
\begin{proof}
The proof roughly follow the strategy in [Cratz-Leon].
%
One can show (see [Taylor-Adler] when $k=d$ or [Azais Wschebor]) that the function $Z_\lambda^u(\phi)$ is almost surely continuous in $u$. It implies the following a.s. convergence
\[Z_\lambda^{u,\varepsilon}(\phi)\underset{\varepsilon\rightarrow 0}{\longrightarrow} Z_\lambda^u(\phi).\]
To show that the convergence also holds in $L^2$, it suffices to show that $\E[Z_\lambda^{u,\varepsilon}(\phi)^2]$ converges to $\E[Z_\lambda^u(\phi)^2]$. From the result of [Gass-Stecconi], the nodal volume is a continuous functional of the covariance function of the field for the $L^2$ norm, and in particular the functions $u\rightarrow \E[Z_\lambda^u(\phi)^2]$ is continuous. From Fatou lemma, one has
\[\E[Z_\lambda^u(\phi)^2]\leq \liminf_{\varepsilon\rightarrow 0} \E[Z_\lambda^{u,\varepsilon}(\phi)^2].\]
Conversely, by Jensen one has
\begin{align*}
\E[Z_\lambda^{u,\varepsilon}(\phi)^2]\leq \int_{\R^k}\E[Z_\lambda^x(\phi)^2]\delta^\varepsilon_u(x)\dd x,
\end{align*}
and the right-hand side converges to $\E[Z_\lambda^u(\phi)^2]$ by continuity, when $\varepsilon$ goes to $0$.
\end{proof}
Let the random field $X = (Y,\nabla Y)$, going from $V$ to the vector space 
\[W = U\times (U\otimes V)\simeq U\otimes (\R\times V),\]
with associated covariance $\Omega$ and spectral measure $\mu$. Let $f:U\mapsto W$ be the tempered distribution, approximated by $f^\varepsilon$, defined for $x\in U$ and $D\in U\otimes V$ as 
\[f^u(x,D) = \delta_u(x)J(D)\quand f^{u,\varepsilon}(x,D) = \delta_u^\varepsilon (x)J(D) \numberthis\label{eq:02}\]
The rescaled nodal volume measure then has the following formal expression
\[Z_\lambda(\phi) = \frac{1}{\lambda^{d/2}}\int_{V} \phi\left(\frac{v}{\lambda}\right)f(X(v))\dd v,\]
which must be understood as the limit in the $L^2$ norm given by Lemma \ref{lemma1}. In particular, the random object $Z_\lambda(\phi)$ formally matches the general setting \eqref{eq:01}.
\subsection{Chaotic expansion}
\label{sec42}
The $L^2$ convergence given by Lemma \ref{lemma12} implies that the function $f$ has a well-defined chaotic expansion of the form
\[\f:x\mapsto \sum_{q=0}^{+\infty} \frac{1}{q!}\langle f_q,H_x^q\rangle,\]
and one can apply the general formalism developed in Section \ref{sec20}. 
To compute the chaotic expansion explicitly, we use the following general lemma.
\begin{lemma}
\label{lemma13}
Let $(X,Y)$ be a jointly Gaussian vector. Assume that $X$ is non-degenerate and let $\rho_X$ be its Gaussian density. Let $h$ be a continuous function with polynomial growth. Then
\[\E[\delta(X)h(X,Y)] := \lim_{\varepsilon\rightarrow 0}\E[\delta^\varepsilon(X)h(X,Y)] = \E[h(0,Y)|X = 0]\rho_X(0).\]
If $X$ and $Y$ are independent then
\[\E[\delta(X)h(X,Y)] = \E[h(0,Y)]\rho_X(0).\]
\end{lemma}
\begin{proof}
The growth condition ensures that all the quantities are well defined. By definition of the conditional expectation
\begin{align*}
\E[\delta^\varepsilon(X)h(X,Y)] &= \int\delta^\varepsilon(u)\E[h(x,Y)|X=x]\rho_X(x)\dd x.
\end{align*}
Gaussian regression implies that the quantity $x\rightarrow \E[h(x,Y)|X = x]$ is well defined for all $x$ and continuous. The conclusion follows by letting $\varepsilon$ tend to zero.
\end{proof}
We state one more lemma that we will use later in the proof.
\begin{lemma}
\label{lemma19}
Let $X$ a standard Gaussian vector of size $n$ and $G$ be a function on $\R^n$ such that for $c\geq 0$ and $m$ an integer such that $G(cX) = c^mG(X)$. Then  
\[\E\left[G(X)\left(\|X\|^2-\E[\|X\|^2]\right)\right] = m\E[G(X)].\]
\end{lemma}
\begin{proof}
By the formula for the moments of a $\chi^2$ distribution, one has
\[\frac{\E\left[(\|X\|^2-\E[\|X\|^2])\|X\|^m\right]}{\E[\|X\|^m]} = m.\]
Using the identity $G(X/\|X\|) = G(X)/\|X\|^m$ and the independence between $X/\|X\|$ and $\|X\|$, we deduce that
\begin{align*}
\E\left[G(X)\left(\|X\|^2-\E[\|X\|^2]\right)\right] &= \E[G(X)]\frac{\E\left[(\|X\|^2-\E[\|X\|^2])\|X\|^m\right]}{\E[\|X\|^m]}\\
&= m\E[G(X)]
\end{align*}
\end{proof}
We now compute quantities related to the second chaotic component $f_2^{(u)}$ of the formal function $f^{u} : (x,D)\mapsto \delta_u(x)J(D)$. We explicit some of the quantities associated to the chaotic projection of order $0,1,2$ and $4$ of the nodal volume of $Y$. Let $\rho$ be the density of the standard Gaussian vector $Y(0)$. We define the positive quantity
\[\alpha_u = \rho(u)\E[J(\nabla Y(0))],\]
and we will omit the index $u$ if $u=0$.
\subsubsection{Chaotic expansion under hypothesis (H1)}
In light of Theorem \ref{thm4} for the limit variance of the second chaos, we compute the spectral measure $\mu$ of the process $X = (Y,\nabla Y)$, and some quantities related to the second chaotic component $f_2^u$ of the tempered distribution $f^u$ defined in \eqref{eq:02} under hypothesis (H1). We recall that the process $X$ takes values in the space
\[W = U\otimes(\R\times V).\]
\begin{lemma}
\label{lemma15}
\[\mu = \psi \,\otimes \left[(1,2i\pi \xi)\otimes (1,-2i\pi \xi)\right].\]
If $U = \R$ then
\[\Imm(\mu) = \enstq{(1,2i\pi\xi )}{\xi\in \supp \psi}\C.\]
\end{lemma}
\begin{proof}
The first point follows from the fact that the Fourier transform derivation into multiplication. In particular, for $q\geq 0$ the  Fourier transform of the differential $D^qr$, is given when defined by $\psi\otimes (2i\pi\xi)^{\otimes q}$. As for the second point, notice that the tensor $(1,2i\pi \xi)\otimes (1,-2i\pi \xi)$ is of rank one and its image is exactly the line directed by the vector $(1,2i\pi \xi)$. The conclusion follows directly from the definition of the image of $\mu$, since $\mu\ll \psi$.
\end{proof}
\begin{lemma}
\label{lemma20}
Under hypothesis (H1) the second chaotic component $f_2^u$ has the expression, in an orthogonal basis on $W$
\[f_2^u = \alpha_u\left[\Id_U \otimes \begin{pmatrix}[c|c]
-1 &  0 \\
\hline
0 & \frac{1}{d}\,\Id_V
\end{pmatrix} + (uu^T)\otimes \begin{pmatrix}[c|c]
1 &  0 \\
\hline
0 & 0
\end{pmatrix}\right].\]
\end{lemma}
\begin{proof}
$Y(0)$ and $\nabla Y(0)$ are independent. Combined with the symmetry of the Jacobian map and Lemma \ref{lemma13}, we deduce that
\[\E[f^{u}(X(0))Y(0)\nabla Y(0)] = 0\quand \E\left[f^u(X(0))(Y(0)Y(0)^T-\Id_U)\right].\]
The invariance by rotation of the Jacobian map implies the existence of a constant $c$ such that 
\[\E\left[f^u(X(0))\left(\nabla Y(0)\otimes \nabla Y(0)^T - \Id_U\otimes\Id_V\right)\right] = c\Id_U\otimes\Id_V.\]
The constant $c$ can be explicitly computed thanks to Lemma \ref{lemma19}, since
\[c = \frac{1}{dk}\E\left[f^u(X(0))\left(\|\nabla Y(0)\|_2^2 - \E[\|\nabla Y(0)\|_2^2]\right)\right] = \frac{\E[f^u(X(0))]}{d} = \frac{\alpha_u}{d}.\]
The conclusion follows.
\end{proof}
The fourth chaotic component $f_4$ of $f$ can be seen as a symmetric bilinear map on $S^2(W^*)$, which can be identified as the direct sum
\[S^2(W^*) = S^2(U)\oplus S^2(U\otimes V)\oplus \left(U\otimes (U\otimes V)\right).\numberthis\label{eq:19}\]
\begin{lemma}
\label{lemma26}
The subspace $U\otimes (U\otimes V)$ of $S^2(W^*)$ is stable by $f_4$. Moreover, the restriction of $f_4$ to this subspace, denoted by $f_4^r$, has expression
\[f_4^r = -\frac{\alpha}{d}\Id_U\otimes \left(\Id_U\otimes \Id_V\right).\]
\end{lemma}
\begin{proof}
The fact that the subspace $U\otimes (U\otimes V)$ directly follows from the symmetry of the Jacobian map. The restriction of $f_4^r$ to this subspace is equal to
\begin{align*}
f_4^r &= \E\left[f(X(0))\left(Y(0)Y(0)^T-\Id_U\right)\otimes\left(\nabla Y(0)\otimes \nabla Y(0)^T - \Id_U\otimes\Id_V\right)\right]\\
&= -\Id_U\otimes \E\left[f(X(0))\left(\nabla Y(0)\otimes \nabla Y(0)^T - \Id_U\otimes\Id_V\right)\right]\\
&= -\frac{\alpha}{d}\Id_U\otimes \left(\Id_U\otimes \Id_V\right),
\end{align*}
where the last line follows from Lemma \ref{lemma20}.
\end{proof}
\subsubsection{Chaotic expansion under hypothesis (H2)}
Under hypothesis (H2), in a good choice of scalar product on $U$ and $V$ one can write $Y = (Y_1,Y_2)$, where $Y_1\perp Y_2$, and $Y_1$ is a real valued Gaussian field such that $\Var(Y(0),\nabla Y(0)) = \Id$. Since $(Y_2,\nabla Y_2(0))$ is not assumed to be a standard Gaussian vector, one must work with a linear transformation of the process $(Y_2,\nabla Y_2)$, so that it is point-wise a standard Gaussian vector.
\begin{lemma}
\label{lemma21}
Under hypothesis (H2), the spectral measure is block diagonal, where the first block corresponds to the spectral measure of $(Y_1,\nabla Y_1)$, and the second block to a linear transformation of the spectral measure of $(Y_2,\nabla Y_2)$. In particular, the first block has the expression
\[\mu_1 = \psi_1\left[(1,2i\pi \xi)\otimes (1,-2i\pi \xi)\right].\]
\end{lemma}
\begin{proof}
The block diagonal property is a direct consequence of the independence, and the expression of the spectral measure of the first bloc follows directly from Lemma \ref{lemma15}.
\end{proof}
\begin{lemma}
\label{lemma22}
Under hypothesis (H2), the second chaotic projection $f_2$ is block diagonal, and the first block has expression
\[f_{2,1} = \alpha  \begin{pmatrix}[c|c]
-1 &  0 \\
\hline
0 & M
\end{pmatrix},\]
where $M$ is a matrix with unit trace.
\end{lemma}
\begin{proof}
The block diagonal property is a direct consequence of the independence of $Y_1$ and $Y_2$ and the symmetry of the Jacobian map. Moreover, one has
\[\E[f(X(0))Y_1(0)\nabla Y_1(0)] = 0\quand \E\left[f(X(0))(Y_1(0)^2-1)\right] = -\alpha.\]
Let 
\[ M = \frac{1}{\alpha}\E[f(X(0))(\nabla Y_1(0)\nabla Y_1(0)^T)-\Id_V)].\]
Then
\[\Tr(M) = \frac{1}{\alpha}\E\left[f(X(0))\left(\|Y_1(0)\|^2-\E[\|Y_1(0)\|^2]\right)\right] = 1,\]
where the last equality follows from Lemma \ref{lemma19} applied to the determinant as a function of the first column.
\end{proof}
\subsubsection{Chaotic expansion under hypothesis (H3)}
Under hypothesis (H3), one has $(Y,\nabla Y) = (\nabla F, \nabla^2 F)$, where $\nabla^2 F$ denotes the Hessian of the random field $F$, which takes values in $S^2(V^*)$, the space of symmetric bilinear map on $V$ equipped with the Frobenius scalar product. We denote by $q$ the covariance function of $F$ and $\omega$ its spectral measure. In that case, the process $(\nabla F, \nabla^2F)$ takes values in the space 
\[W = V\times S^2(V^*).\]
 On can assume that the scalar product defined on $V$ is such that $\nabla F$ is point-wise a standard Gaussian vector, and after rescaling, that $F$ is point-wise a standard Gaussian vector. In general $\nabla^2 F$ is not a standard Gaussian vector, and one must work with a linear transformation, denoted by $TF$, such that $TF$ is a standard Gaussian vector on $\Sym(V)$. Note that the distribution of $TF$ is exactly $\frac{1}{\sqrt{2}}\GOE$, where $\GOE$ is the Gaussian Orthogonal Ensemble. If $F$ has an isotropic distribution (i.e. $q$ is a radial function), then the linear transformation is explicit and is the object of the next Lemma \ref{lemma25}.  In that case we define, for $v,w$ two arbitrary unit orthogonal vectors the quantity
\[\beta = \E[\nabla^2 F(v,w)^2],\]
which does not depend of the choice of $v,w$ by the invariance by rotation, and the constants
\[\beta_0 = \frac{d}{d+2},\quand \gamma =1\pm \sqrt{\frac{2}{d+2}},\]
where the choice $\pm$ is arbitrary.
\begin{lemma}
\label{lemma14}
Assume that $F$ has an isotropic distribution. For orthogonal unit vectors $v,w\in V$,
\[\E[\nabla^2 F(v,v)^2] = 3\beta,\quad \E[\nabla^2 F(v,w)^2] = \E[\nabla^2 F(v,v)\nabla^2 F(w,w)] = \beta.\]
Moreover, 
\[\beta\geq \beta_0\]
with equality if and only if $F$ is the isotropic random wave model.
\end{lemma}
\begin{proof}
In the following, $v,w$ are orthogonal unit vectors in $V$. By symmetry of the derivatives, one has
\[\E[\nabla^2 F(v,w)^2] = \partial_v\partial_v\partial_w\partial_w q(0) = \E[\nabla^2 F(v,v)\nabla^2 F(w,w)].\]
Using the invariance by rotation
\begin{align*}
E[\nabla^2 F(v,v)^2] &= \frac{1}{2}\E\left[\nabla^2 F\left(\frac{v+w}{\sqrt{2}},\frac{v+w}{\sqrt{2}}\right)^2\right] + \frac{1}{2}\E\left[\nabla^2 F\left(\frac{v-w}{\sqrt{2}},\frac{v-w}{\sqrt{2}}\right)^2\right]\\
&= \frac{1}{4}\left(\E[\nabla^2 F(v,v)^2] + \E(\nabla^2 F(w,w)^2]\right) + \frac{1}{2}\E[\nabla^2 F(v,v)\nabla^2 F(w,w)] + \E[\nabla^2 F(v,w)^2]
\end{align*}
We deduce that
\[E[\nabla^2 F(v,v)^2] = \E[\nabla^2 F(v,v)\nabla^2 F(w,w)] + 2\E[\nabla^2 F(v,w)^2] = 3\beta.\]
Moreover, a quick computation shows that
\[\E\left[\left(F-\frac{1}{d}\Delta F\right)^2\right] = \beta\frac{d+2}{d}-1.\]
We deduce that $\beta\geq \beta_0$ with equality if and only if $F$ is the random wave model. 
\end{proof}
\begin{lemma}
\label{lemma25}
Assume that $F$ has an isotropic distribution. Then the vector
\[TF = \frac{1}{\sqrt{2\beta}}\left(\nabla^2 F - \frac{\gamma}{d}(\Delta F)\Id_V\right)\]
is a standard Gaussian vector. 
\end{lemma}
\begin{proof}
The constant $\gamma$ must satisfies the identities
\[0= \E\left[\left(\nabla^2 F(v,v)-\gamma\frac{\Delta F}{d}\right)\left(\nabla^2 F(w,w)-\gamma\frac{\Delta f}{d}\right)\right] = \beta\left[1 - \frac{2(d+2)}{d}\gamma + \frac{(d+2)}{d}\gamma^2\right],\]
and 
\[ 1 = \frac{1}{2\beta}\E\left[\left(\nabla^2 F(v,v)-\gamma\frac{\Delta F}{d}\right)^2\right] = \frac{1}{2}\left[3 - \frac{2(d+2)}{d}\gamma + \frac{(d+2)}{d}\gamma^2\right].\]
We deduce that $\gamma$ is a root of an explicit second degree polynomial, whose roots are given by
\[\gamma  = 1 \pm\sqrt{\frac{2}{d+2}}.\]
\end{proof}
\begin{lemma}
\label{lemma23}
Under hypothesis (H3), the spectral measure of $(\nabla F, \widetilde{\nabla^2 F})$ can be written as
\[\mu = \omega\left[2i\pi\xi,S(\xi)\right]\otimes \left[-2i\pi\xi,\overline{S(\xi)}\right],\]
where the map $S(\xi)$ is a complex quadratic form in the variable $\xi$. If the distribution of $F$ is isotropic, then the spectral measure of $(\nabla F,TF)$ is
\[\mu = \omega\left[2i\pi\xi,-\frac{4\pi^2}{\sqrt{2\beta}}\left(\xi\otimes\xi-\frac{\gamma}{d}\|\xi\|^2\Id\right)\right]\otimes \left[-2i\pi\xi, -\frac{4\pi^2}{\sqrt{2\beta}}\left(\xi\otimes\xi-\frac{\gamma}{d}\|\xi\|^2\Id\right)\right].\]
\end{lemma}
\begin{proof}
Both expression follows directly from Lemma \ref{lemma15} and Lemma \ref{lemma25}.
\end{proof}
\begin{lemma}
\label{lemma24}
Under hypothesis (H3), the second chaotic projection $f_2$ has expression
\[f_{2} = \alpha  \begin{pmatrix}[c|c]
-\Id_V &  0 \\
\hline
0 & M
\end{pmatrix}.\]
If the distribution of $F$ is isotropic, then the bilinear map $f_2$ has expression, for $v\in V$ and $S$ a symmetric bilinear map
\[f_2(v,S) = -\alpha\|v\|^2 + A\Tr(S^2) + B\Tr(S)^2.\]
Moreover
\[A\left(1-\frac{2\gamma}{2}+\frac{\gamma^2}{d}\right) + B\left(1-\gamma\right)^2 = \frac{2\beta_0}{d}\alpha.\]
\end{lemma}
\begin{proof}
The first property is a consequence of the independence between $\nabla F$ and $TF$. As for the second one, we define the constants, for $v,w$ orthogonal unit vectors in $V$
\[A = \E\left[f(X)\left(2TF(v,w)^2-1\right)\right],\quad B = \E\left[f(X)TF(v,v)TF(w,w)\right]\]
and
\[C = \E\left[f(X)\left(TF(v,v)^2-1\right)\right].\]
Using the invariance by rotation one has, following the proof of Lemma \ref{lemma25}, the identity
\[C = A + B.\numberthis\label{eq:18}\]
Since the distribution of $X$ in invariant by rotation, $f_2$ is a quadratic form on the space of symmetric bilinear forms which is invariant by rotation. It must then by a symmetric quadratic form on the eigenvalues, and must then be a linear combination of the quadratic forms $\Tr(S^2)$ and $\Tr(S)^2$. The identification of the coefficient $A$ and $B$ follows by a direct computation of the quadratic form, for instance by fixing an orthogonal basis on $V$.\jump

It also follows from Lemma \ref{lemma19} and Equation \eqref{eq:18} that
\[d\alpha = \E[f(X)(\|TF\|^2-\E[\|TF\|^2])] = \frac{d(d-1)}{2}A + dC = \frac{d(d+1)}{2}A + dB.\]
The conclusion follows by observing the identities
\[A\left(1-\frac{2\gamma}{d}+\frac{\gamma^2}{d}\right) + B\left(1-\gamma\right)^2 = \frac{d+1}{d+2} A + \frac{2}{d+2}B = \frac{2}{d+2}\alpha = 
\frac{2\beta_0}{d}\alpha.\]
\end{proof}
As in the subsection concerning hypothesis (H1), the fourth chaotic component $f_4$ of $f$ can be seen as a symmetric bilinear map on $S^2(W^*)$, which can be decomposed as the direct sum
\[S^2(W^*) = S^2(V)\oplus S^2(S^2(V))\oplus \left(V\otimes S^2(V)\right).\numberthis\label{eq:21}\]
\begin{lemma}
\label{lemma27}
The subspace $U\otimes S^2(V)$ of $S^2(W^*)$ is stable by $f_4$. Moreover, the restriction of $f_4$ to this subspace, denoted by $f_4^r$, has expression, for $v\in V$ and $S$ a symmetric bilinear map
\[f_4^r(v\otimes S) = \|v\|^2\left(A\Tr(S^2) + B\Tr(S)^2\right),\]
where $A,B$ are defined in Lemma \ref{lemma24}.
\end{lemma}
\begin{proof}
The fact that the subspace $U\otimes (U\otimes V)$ directly follows from the symmetry of the Jacobian map. The restriction of $f_4^r$ to this subspace is equal to
\begin{align*}
f_4^r &= \E\left[f(X(0))\left(\nabla F(0)\nabla F(0)^T-\Id_V\right)\otimes\left(TF(0)\otimes TF(0) - \Id_V\otimes\Id_V\right)\right]\\
&= -\Id_V\otimes \E\left[f(X(0))\otimes\left(TF(0)\otimes TF(0) - \Id_V\otimes\Id_V\right)\right]\\
&= -\Id_V\otimes f_2(0,\,\cdot\,).
\end{align*}
The conclusion follows from the expression of $f_2$ given by Lemma \ref{lemma24}.
 \end{proof}
\subsection{Self-convolution of hypersurface measures}
\label{sec43}
The goal of this subsection is to derive formulas for the convolution of continuous measures supported on the sphere. These computations could certainly be approached from a probabilistic perspective, specifically by considering the distribution of a \textit{uniform} random walk, for which precise results concerning the probability of returning close to zero after a small number of steps exist; see, for instance, \cite{Bor12}. We define the convolution symbol $\str$ as
\[\mu\str\nu = \mu\str(\nu(-\,\cdot)).\]
If the measure $\nu$ is symmetric w.r.t. the origin, it coincides with the usual convolution operation. We say that a function $f$ defined on the ball $B(0,R)$ in $\R^d$ is \textit{polar-continuous} if the function $\tilde{f}:(r,\theta)\mapsto f(r\theta)$ is continuous on $[0,R]\times \Ss^{d-1}$. The typical example of a polar continuous (but not continuous) function is the mapping $x\mapsto x/\|x\|$.
\begin{lemma}
\label{lemma16}
Let $\sigma$ be the uniform probability measure on the sphere of radius $1$. Then $\sigma*\sigma$ has a density w.r.t the Lebesgue measure given by
\[\sigma*\sigma(x) = C_d \frac{1}{\|x\|}(4-\|x\|^2)^{\frac{d-3}{2}}\one_{B(0,2)}(x),\]
for some normalization constant $C_d$. If $\rho_1,\rho_2$ are continuous functions on the sphere, then the measure $(\rho_1\sigma)\str(\rho_2\sigma)$ has a polar-continuous density w.r.t the measure $\sigma*\sigma$. Precisely,
 if $\mu_\theta$ designs the probability measure of the $d-2$ dimensional sphere of radius $1$ and normal vector $\theta$, then
\[\lim_{(r,\theta)\rightarrow (0,\theta_0)} \frac{\dd (\rho_1\sigma)\str(\rho_2\sigma)}{\dd \sigma*\sigma}(r,\theta) = \int \rho_1(y)\rho_2(y)\dd\mu_{\theta_0}(y),\]
\end{lemma}
\begin{proof}
Let $M$ and $N$ two compact hypersurface of $\R^d$ with respective Lebesgue measure $\sigma_M$ and $\sigma_N$. Let $\rho_M$ and $\rho_N$ be two smooth functions on $M$ and $N$, respectively. We assume that $M$ and $N$ are transverse, meaning that for all $x\in M\cap N$, the unit normal vectors of the two hypersurfaces at point $x$ form an angle $\theta_x\in ]0,\pi[$. By compactness, there is $\eta>0$ such that for all $u\in B(0,\eta)$, $M+u$ is also transverse to $N$. Moreover, $M\cap N$ is a submanifold of dimension $d-2$, and by considering a $\varepsilon$-neighborhood of $M$ and $N$ respectively, one can observe, taking the limit as $\varepsilon$ goes to $0$, that $(\rho_M\sigma_M)\str(\rho_N\sigma_N)$ has a density in $B(0,\eta)$ with respect to the Lebesgue measure on $\R^d$ such that
\[(\rho_M\sigma_M)\str(\rho_N\sigma_N)(0) = \int_{M\cap N} \frac{\rho_M(x)\rho_N(x)}{|\sin(\theta_x)|}\dd \sigma_{N\cap M}(x).\numberthis\label{eq:12}\]
A bit of geometry shows that two unit spheres whose centers are shifted by a vector $x$ with norm $r\in ]0,2[$ intersects everywhere at an angle $\theta_r$ such that
\[\sin\left(\frac{\theta_r}{2}\right) = \frac{r}{2},\qsothat \sin(\theta_r) = 2r\sqrt{1-\left(\frac{r}{2}\right)^2}.\]
The intersection of the two spheres is a $d-2$ dimensional sphere $B_x$ of radius $\sqrt{1-\left(\frac{r}{2}\right)^2}$. Gathering this fact and Formula \eqref{eq:12}, we deduce that the convolution $\sigma*\sigma$ has a density w.r.t. Lebesgue measure given by
\[\sigma*\sigma(x) = C_d \frac{1}{\|x\|}(4-\|x\|^2)^{\frac{d-3}{2}}\one_{B(0,2)}(x),\]
and more generally
\[(\rho_1\sigma)\str(\rho_2\sigma)(x) = \left(\int_{B_x}\rho_1(y)\rho_2(y)\dd \mu_x(y)\right)\sigma*\sigma(x),\]
where $\mu_x$ is the uniform probability measure on the $d-2$ dimensional sphere $B_x$. Passing in polar coordinates $x = (r,\theta)$, the sphere $B_x$ converges, as $r$ goes to $0$ and $\theta$ goes to $\theta_0$, to the $d-2$ sphere $B_{\theta_0}$ of radius $1$ and normal vector $\theta_0$, and
\[\lim_{(r,\theta)\rightarrow (0,\theta_0)} \int_{B_x}\rho_1(y)\rho_2(y)\dd \mu_x(y) = \int_{B_{\theta_0}}\rho_1(y)\rho_2(y)\dd\mu_{\theta_0}(y).\]
\end{proof}
\begin{remark}
\label{rem2}
The key assumption regarding the sphere is that it is everywhere \textit{positively curved}. The curvature implicitly appears as the differential of the normal vector at the start of the proof. A similar degeneracy in the function $\sigma*\sigma$ could also be observed for more general positively curved manifolds $M$, near the point $x = 0$ and along the boundary of the set of points $x$ such that $M \cap (M + x) \neq \emptyset$. By linking the regularity of a function with the decay of its Fourier transform, this provides an alternative heuristic explanation —beyond the usual stationary phase approach— of why the Fourier transform of a set with positively bounded curvature exhibits favorable decay properties (see the definition of test functions in Section \ref{sec11}).
\end{remark}
\begin{lemma}
\label{lemma17}
Let $\sigma$ be the uniform probability measure on the sphere of radius $1$. Then the measure $\sigma^{*4}$ has a density w.r.t the Lebesgue measure such that for some positive constant $C_d$,
\[\sigma^{*4}(x) \simeq \left\lbrace\begin{array}{llll}
C_2|\log(\|x\|)|\;\text{ if } d=2,\\
C_d\;\qquad\quad\quad\;\text{ if } d\geq 3,\\
\end{array}
\right.\]
In dimension $2$, let $\omega_1$ and $\omega_2$ be polar continuous functions on the closed ball $B(0,2)$. Then there are \textbf{positive} constants $C$, $C'$ such that
\[\lim_{x\rightarrow 0} \frac{1}{|\log(\|x\|)|}(\omega_1\sigma^{*2})\str (\omega_2\sigma^{*2})(x) = C\int_{\Ss^1}\omega_1(0,\theta)\omega_2(0,\theta)\dd\theta + C'\int_{\|y\|=2}\omega_1(y)\omega_2(y)\dd y.\]
\end{lemma}
\begin{proof}
The previous Lemma \ref{lemma16} implies for $d\geq 3$ that the measure $\sigma*\sigma$ is square integrable. In particular, the measure $\sigma^{*4}$ has a continuous density, and $\sigma^{*4}(0) = \|\sigma*\sigma\|_2^2>0$. When $d=2$, the function $\sigma*\sigma$ is not square-integrable. One has
\[(\omega_1\sigma)\str(\omega_2\sigma)(x) = C\int_{\R^d}\frac{\one_{B(0,2)}(x+y)\one_{B(0,2)}(y)\omega_1(x+y)\omega_2(y)}{\|y\|\|x+y\|\sqrt{4-\|y\|^2}\sqrt{4-\|x+y\|^2}}\dd y.\]
We must understand the singular loci : near the point $y=0$, and near the sphere $\|y\|=2$. Let $x=(\varepsilon,\theta)$. Near $y=0$ we make the substitution $y = \varepsilon u$. 
The singularity comes from the term $1/\|u\|\|\theta+ u\|$ that integrates as a $|\log(\varepsilon)|$ on $B(0,2/\varepsilon)$. One can then isolate the mass of the singularity for instance on the domain 
\[D_\varepsilon = B\left(0,\frac{2}{\varepsilon|\log(\varepsilon)|}\right)\setminus B(0,|\log(\varepsilon)|).\]
On this domain one has $\theta+u\simeq u$ and $\varepsilon u \rightarrow 0$, so that as $\varepsilon$ goes to $0$,
\[(\rho_1\sigma)\str(\rho_2\sigma)(x) \simeq C\int_{D_\varepsilon}\frac{\omega_1(\varepsilon u)\omega_2(\varepsilon u)}{4\|u\|^2}\dd y\simeq  C|\log(\varepsilon)|\int_{\Ss^1}\omega_1(0,\theta)\omega_2(0,\theta)\dd\theta.\]
By a similar argument one can obtain the same logarithmic scale near the sphere $\|y\|=2$.
\end{proof}
\section{Proofs of the main theorems}
\label{sec50}
In this section, we prove the main theorem stated in introduction. The first section is devoted to the proof of the general theorems, which are relatively straightforward consequences of the lemmas stated in Section \ref{sec30}. The second subsection exposes the proof related to the nodal measure, which follow from the material developed in Section \ref{sec40}. In the following, we will use the notation developed in the previous subsection.
\subsection{Proofs of the general theorems}
\label{sec51}
In this section, we prove the general theorems stated in introduction, namely \ref{thm1}, Corollaries \ref{cor1} and \ref{cor2} and Theorem \ref{thm4}. For $u,v\in V$, one has from Equation \eqref{eq:13}
\[\Var(f(X(u)),f(X(v))) = \sum_{q=1}^{+\infty} \E[\pi_q(f)(X(u))\pi_q(f)(X(v))] =  \sum_{q=1}^{+\infty} \frac{1}{q!}(\Omega(u-v))^{\otimes q}(f_q,f_q).\]
One deduce that
\begin{align*}
\Var(Z_\lambda(\phi))&= \sum_{q=1}^{+\infty} \frac{1}{\lambda^d}\int_{V^2}\phi\left(\frac{u}{\lambda}\right)\phi\left(\frac{v}{\lambda}\right)\E[\pi_q(f)(X(u))\pi_q(f)(X(v))]\dd u\dd v\\
& = \sum_{q=1}^{+\infty} \Var(Z_\lambda^{(q)}(\phi)),
\end{align*}
proving the first part of the lemma. As for the second part, the stationarity implies the relation
\[\Var(Z_\lambda^{(q)}(\phi)) = \frac{1}{q!}\int_V (\Omega(u-v))^{\otimes q}(f_q,f_q)\left(\frac{1}{\lambda^d}\phi\left(\frac{\,.\,}{\lambda}\right)*\phi\left(\frac{-\,.\,}{\lambda}\right)\right)(v)\dd v\]
Using Plancherel theorem, one obtain
\begin{align*}
\Var(Z_\lambda^{(q)}(\phi)) = \frac{1}{q!}\int_{V^*} \lambda^d |\hat{\phi}(\lambda\xi)|^2\dd\widehat{\Omega^{\otimes q}(f_q,f_q)}(\xi) = \frac{1}{q!}\int_{V^*}\gamma_\lambda(\xi)\,\dd \mu^{*q}(f_q)(\xi).
\end{align*}
The proof of Corollary \ref{cor1} is a direct application of Lemma \ref{lemma7} applied to the measure $\mu^{*q}(f_q)$. As for the next Corollary \ref{cor2}, the integrability assumption on the density $\Sigma$ implies that $\Sigma^{*q}$ is a bounded and continuous function. So that
\[\lim_{\lambda\rightarrow +\infty} \Var(Z_\lambda^{(q)}(\phi)) = \Sigma^{*q}(f_q)(0).\]
The conclusion follows from the convolution formula
\[\Sigma^{*q}(0) = \int_{\xi_1+\ldots+\xi_q=0} \bigotimes_{k=1}^q \Sigma(\xi_k)\dd\xi.\]
When $q=1$ and $q=2$, these expressions reduces to the formulas  stated.
\subsection{Proofs of the variance of the nodal measure}
\label{sec52}
The first section is devoted to the proof of Theorem \ref{thm6}. The second subsection is devoted to the proof Theorem \ref{thm2} and Theorem \ref{thm3}. The last subsection is devoted to the proof of Theorem \ref{thm5}.
\subsubsection{Proof of Theorem \ref{thm6}}
The proof of the asymptotic normality exactly follows the same strategy as \cite{Nua20}, and is based on the Fourth moment theorem \cite{Nua05}. The argument for the cancellation of the contractions is completely identical. The only difference is that in our setting, $f$ is a tempered distribution and one cannot use directly the usual Arcones inequality \cite{Arc94} to upper bound the variance of $Z_\lambda(\phi)$, and we detail here the proof for the variance bound. We use the notations of Section \ref{sec23}. By assumption on the field, one has $\Omega(v)\in E^*(\eta)$, where $\eta$ is the Gaussian measure on $W$. We define, for $v\neq 0$ the quantity 
\[\E[f(0)f(v)] = (f\otimes f)(\rho_\Omega).\]
It follows from \cite{Gas24} applied with $p=2$, or directly from the standard divided difference trick presented in the introduction of that paper, that there is a constant $C$ such that for $v\in B(0,1)$
\[\E[f(0)f(v)]\leq C.\numberthis\label{eq:15}\]
By Kac--Rice formula
\begin{align*}
\Var(Z_\lambda(\phi)) &= \int_{V\setminus \{0\}}\frac{1}{\lambda^d}\phi*\phi\left(\frac{v}{\lambda}\right)\E\left[f(X(0))f(X(v))\right]\dd v + \E[Z_\lambda(\phi)]^2\one_{k=d}
\end{align*}
Since $\Omega$ converges to $0$ as $\|v\|\rightarrow +\infty$, by compactness one can find a compact $A$ of $E^*(\eta)$ such that for all $v\in V\setminus B(0,1)$, one has $\Omega(v)\in A$. Then by Equation \eqref{eq:14}, there is a constant $C_q$ such that for $v\in V\setminus B(0,1)$,
\[\left|\E\left[f(X(0))f(X(v))\right] - \sum_{k=0}^{q-1}\E\left[f_k(X(0))f_k(X(v))\right]\right|\leq C_q\|\Omega\|^q.\numberthis\label{eq:16}\]
The function $\frac{1}{\lambda^d}\phi*\phi(\cdot/\lambda)$ uniformly converges to the constant function $1$, and the convergence of the variance of $Z_\lambda^{(q+)}(\phi)$ follows by dominated convergence, gathering Equation \eqref{eq:15} and \eqref{eq:16}.
\subsubsection{Proof of Theorem \ref{thm7}}
Using the independence between $Y(0)$ and $\nabla Y(0)$ and the symmetry of the Jacobian map, the first chaotic projection for the nodal volume has the expression
\[f_1 = \alpha_u (u, 0).\]
Using Corollary \ref{cor2} with $q=1$ we deduce that
\[\lim_{\lambda\rightarrow +\infty}\Var(Z_\lambda^{u,(1)}(\phi)) = \alpha_u\Sigma(0)(u,u).\]
If $r$ is integrable, the density $\Sigma$ is well-defined and continuous. The CLT for $Z_\lambda$ follows from Theorem \ref{thm6} applied to $q=1$, and the lower bound for the variance given by the variance of first chaotic component.
\subsubsection{Proof of Theorem \ref{thm2} and Theorem \ref{thm3}}
Let $\Sigma = \widehat{\Omega}$ be the spectral density of the random process $X = (Y,\nabla Y)$ that takes values in the space $W = U\times (U\otimes V)\simeq U\otimes (\R\times V)$. It follows from Corollary \ref{thm2} that the limit expression for the variance is given by the quantity
\[\lim_{\lambda\rightarrow +\infty}\Var(Z_\lambda^{(2)}(\phi)) = \frac{1}{2}\int_{V^*}\Tr(f_2\Sigma(\xi)f_2\Sigma(\xi))\dd \xi.\]

\noindent \underline{$\bullet$ Assume the hypothesis (H1)}: \jump
Using the definition of the trace, or equivalently the mixed product property, and Lemma \ref{lemma2} applied to the rank one matrix $(1,2i\pi \xi)\otimes (1,-2i\pi \xi)$, one gets thanks to Lemma \ref{lemma15} and Lemma \ref{lemma20}
\begin{align*}
\Tr\left(f_2\Sigma(\xi)f_2\Sigma(\xi)\right) &= \Tr(\sigma(\xi)^2)\left((1,2i\pi\xi)^T\begin{pmatrix}[c|c]
-\alpha &  0 \\
\hline
0 & \frac{\alpha}{d}\,\Id
\end{pmatrix}(1,-2i\pi\xi)\right)^2\\
&= \alpha^2\left(-1+\frac{4\pi^2}{d}\,\|\xi\|^2\right)^2\|\psi(\xi)\|_2^2. \numberthis\label{eq:04}
\end{align*}
Since the function $\psi$ is in $L^2$ and is non identically zero, it cannot be supported on the zero set of a non-zero quadratic equation. Therefore, this last quantity is not the zero function and 
\[\lim_{\lambda\rightarrow +\infty}\Var(Z_\lambda^{(2)}(\phi)) >0.\]
Note that the proof didn't make use of Theorem \ref{thm4} because the computation are explicit. If we remove the hypothesis of square integrability on $\mu$, then let $\dd\psi = \varsigma(\xi)\dd\nu$ be a representation of $\sigma$. We deduce the following representation for the spectral measure $\mu$
\[\dd\mu = \Theta(\xi)\dd\nu,\qwith \Theta(\xi) = \varsigma(\xi)\,\otimes \left[(1,2i\pi \xi)\otimes (1,-2i\pi \xi)\right]\dd \nu.\]
In the case $k=1$ the isotropic cone of $f_2$ and the image of $\mu$ are given explicitly from Lemma \ref{lemma15} by
\[C(f_2) = \enstq{(1,2i\pi\xi)}{\|\xi\|=d}\C \quand \Imm(\mu) = \enstq{(1,2i\pi\xi)}{\xi\in \supp \sigma}\C,\]
and $\Imm(\mu)\subset C(f_2)$ if and only if the measure $\sigma$ is supported on the sphere of radius $\sqrt{d}/2\pi$, i.e. if $Y$ is a random wave. The conclusion of Theorem \ref{thm3} directly follows from the two points of Theorem \ref{thm4}, applied to the sphere as a smooth compact hypersurface of $V^*$. In the case $k>1$, the computation of the isotropic cone and the image are not convenient and we follow a more direct approach. The same computation as in Equation \eqref{eq:04} shows that
\[\Tr\left(f_2\varsigma(\xi)f_2\varsigma(\xi)\right) = \alpha^2\left(-1+\frac{4\pi^2}{d}\,\|\xi\|^2\right)^2\|\varsigma(\xi)\|_2^2. \numberthis\label{eq:05}\]
If $Y$ is not a random wave then this quantity is positive on a set of positive $\nu$-measure, and the conclusion follows from the proof Lemma \ref{lemma10}. Conversely, if $Y$ is a regular random wave then this quantity cancels for all $\xi$ in the support of $\nu$ and the conclusion of Theorem \ref{thm3} follows from the proof of Lemma \ref{lemma11}, and in particular from the second point.\jump

\noindent \underline{$\bullet$ Assume the hypothesis (H2)}: \jump
It follows from Lemma \ref{lemma21} and Lemma \ref{lemma22} that
\[\Tr\left(f_2\Sigma(\xi)f_2\Sigma(\xi)\right) \geq \Tr\left(f_{2,1}\Sigma_1(\xi)f_{2,1}\Sigma_1(\xi)\right) \geq \alpha^2 \left(-1+\frac{4\pi^2}{d}\,\xi^TM\xi\right)^2|\psi_1(\xi)|^2,\]
and the conclusion follows in a similar fashion as in the previous case.\jump

\noindent \underline{$\bullet$ Assume the hypothesis (H3)}: \jump

Without isotropy assumption, \ref{lemma25} and Lemma \ref{lemma24} implies that
\[\Tr\left(f_2\tilde{\Sigma}(\xi)f_2\tilde{\Sigma}(\xi)\right) = \omega(\xi)^2\alpha^2\left(-4\pi^2\|\xi\|^2  + S(\xi)M\overline{S(\xi)}\right)^2.\]
The conclusion follows in a similar fashion as in the previous case. If the field $F$ has an isotropic distribution, then one can explicit this quantity, thanks to Lemma \ref{lemma25}. One has
\begin{align*}
\Tr\left(f_2\tilde{\Sigma}(\xi)f_2\tilde{\Sigma}(\xi)\right) &= \omega(\xi)^2f_2\left(2i\pi\xi, -\frac{4\pi^2}{\sqrt{2\beta}}\left(\xi\otimes\xi-\frac{\gamma}{d}\|\xi\|^2\Id\right)\right)^2\\
 &= \omega(\xi)^2\left[-4\pi^2\|\xi\|^2\alpha + \frac{(4\pi^2)^2\|\xi\|^2}{2\beta}\left[A\left(1-\frac{2\omega}{2}+\frac{\omega^2}{d}\right) + B\left(1-\omega\right)^2\right]\right]^2\\
&= \hat{q}(\xi)^2(4\pi^2\|\xi\|^2\alpha)^2\left[-1 + \frac{4\pi^2\|\xi\|^2}{d}\frac{\beta_0}{\beta}\right]^2\numberthis\label{eq:22}\\
\end{align*}
The conclusion of Theorem \ref{thm3} follows from the same argument as in the case of hypothesis (H1), using the fact that $\beta=\beta_0$ if and only if $F$ is the isotropic random wave model.\jump

\noindent \underline{$\bullet$ Conclusion}: \jump

The Hermite rank of the nodal volume functional is two, and $r\in L^2(V)$. It follows from Theorem \ref{thm6} applied with $q=2$ that the limiting variance of $Z_\lambda$ exists, is finite and under any of the three hypotheses of Lemma \ref{thm2}, is positive thanks to the lower bound
\[\Var(Z_\lambda(\phi))\geq \Var(Z_\lambda^{(2)}(\phi))>0.\]
The CLT for $Z_\lambda(\phi)$ again follows from Theorem \ref{thm6}.\jump

\begin{remark}
\label{rem3}
For $u$-level with $u\neq 0$, one can prove by similar arguments that there is no second chaos cancellation for the isotropic random waves model. Indeed,  Equation \eqref{eq:05} becomes,
\[\Tr\left(f_2\Sigma(\xi)f_2\Sigma(\xi)\right) = \alpha^2\left(\|u\|^2-1 + \frac{4\pi^2}{d}\|\xi\|^2\right)^2 + (k-1)\alpha^2\left(-1 + \frac{4\pi^2}{d}\|\xi\|^2\right)^2 >0.\]
Since the measure $\mu$ is not square integrable, the limiting variance explodes and one can explicitly compute the asymptotics, as it has been done in \cite{Ros16}, and prove the CLT by standard methods.
\end{remark}
\subsubsection{Proof of Theorem \ref{thm5}}

\noindent \underline{$\bullet$ Assume that $Y$ is a symmetric regular random wave}: \jump

Let $\mu$ be the spectral measure of the process $(Y,\nabla Y)$ taking values in the space of Hermitian forms on $W = U\otimes(\R \times V)$. Let $\varsigma$ be the density of the spectral measure $\psi$ with respect to the spherical measure $\sigma$. Since the covariance function of $Y$ is symmetric, the spectral measure $\psi$ has no imaginary component and is thus a symmetric measure (w.r.t $\xi$), so is the density $\varsigma$.\jump 

We want to apply Theorem \ref{thm4} with $q=2$, and conclude that the variance converges to a limit positive constant, with a logarithmic correction in the case $d=2$. From Lemma \ref{lemma15}, the spectral measure $\mu$ has a density w.r.t the spherical measure $\sigma$ given by
\[\dd\mu = \varsigma(\xi)\otimes \left[(1,2i\pi \xi)\otimes(1,-2i\pi \xi)\right]\dd\sigma,\]
Thanks to Lemma \ref{lemma16}, the measure $\mu*\mu$ has a density $\Psi$ w.r.t the measure $\sigma*\sigma$, which is bounded and polar-continuous. The variance of the fourth chaos has the expression
\begin{align*}
\Var(Z_\lambda^{(4)}(\phi)) &= \frac{1}{24}\int_{V}\left(|\hat{\phi}|^2\right)_\lambda(\xi)\,\dd (\mu*\mu)^{*2}(f_4)(\xi)\\
&= \frac{1}{24}\int_{V}\left(|\hat{\phi}|^2\right)_\lambda(\xi)\,\dd \left[\Psi(\sigma^{*2})*\Psi(\sigma^{*2})\right](f_4)(\xi)
\end{align*}
The variance asymptotics is explicit thanks to Lemma \ref{lemma17}. When $d=2$, we have by integration of equivalences,the expression, for positive constants $C,C'$
\[\lim_{\lambda\rightarrow +\infty}\frac{\Var(Z_\lambda^{(4)}(\phi))}{\log(\lambda)} = C\int_{\Sp^1}\Tr(f_4\Psi(0,\theta)f_4\Psi(0,\theta))\dd\theta + C'\int_{\|\xi\|=\frac{\sqrt{d}}{\pi}}\Tr(f_4\Psi(\xi)f_4\Psi(\xi))\dd \xi.\]
When $d\geq 3$, the limit variance is given by Corollary \ref{cor2} and has expression
\[\lim_{\lambda\rightarrow +\infty} \Var(Z_\lambda^{(4)}(\phi)) = C\int_{V}\Tr(f_4\Psi(\xi)f_4\Psi(\xi))\left(\frac{\dd\sigma*\sigma}{\dd\xi}(\xi)\right)^2\dd\xi.\]
In both cases, we can conclude that the limit constant is positive if we are able to show that for $\theta\in \Sp^{d-1}$ the quantity $\Tr(f_4\Psi(0,\theta)f_4\Psi(0,\theta))$ is positive. According to Lemma \ref{lemma16},
\[\Psi(0,\theta) = \int_{\Sp^{d-1}} \left(\varsigma(\xi)\otimes \left[(1,2i\pi \xi)\otimes(1,-2i\pi \xi)\right]\right)^{\otimes 2}\dd \sigma_\theta(\xi),\numberthis\label{eq:20}\]
where $\sigma_\theta$ is the uniform probability measure on the $d-2$-sphere on $V$ with normal vector $\theta$. The map $\Psi(0,\theta)$ can be seen as symmetric bilinear map over $S^2(W^*)$, as the fourth chaotic component $f_4$, see the paragraph above Lemma \ref{lemma26}. In particular, it follows from the symmetry of the integrand in Equation \eqref{eq:20} that the component associated to the subspace $U\otimes (U\otimes V)$ of $S^2(W^*)$ defined in Equation \eqref{eq:19} is stable by $\Psi(0,\theta)$. The restriction of $\Psi(0,\theta)$ to this subspace, denoted $\psi^r(0,\theta)$, is given by
\[\Psi^r(0,\theta) = 4\pi^2\int_{\Sp^{d-1}} \varsigma(\xi)\otimes \varsigma(\xi) \otimes\left(\xi\otimes\xi\right) \dd \sigma_\theta(\xi).\]
We then have , given the expression of $f_4^r$ given by Lemma \ref{lemma26}, and the mixed product property
\begin{align*}
\Tr(f_4\Psi(0,\theta)f_4\Psi(0,\theta)) &\geq \Tr(f_4^r\Psi^r(0,\theta)f_4^r\Psi^r(0,\theta))\\
&\geq \left(\frac{\alpha}{d}\right)^2(4\pi^2)^2\int_{\Sp^{d-1}}\int_{\Sp^{d-1}}\Tr(\varsigma(\xi)\varsigma(\eta))^2\langle \sigma,\eta\rangle^2\dd \sigma_\theta(\xi)\dd \sigma_\theta(\eta)
\end{align*}
The integrand is continuous, and positive for $\xi=\eta$. We conclude that the whole integral is positive, and so is the limit variance. \jump

To prove the asymptotic normality, note when $d\geq 3$ that $\Omega\in L^4$. It follows from Theorem \ref{thm6} that the variance of $Z_\lambda^{(4+)}(\phi)$ converges to a finite constant, which is positive thanks to the previous computation , and that $Z_\lambda^{(4+)}(\phi)$ satisfies a CLT. Since the variance of $Z_\lambda^{(2)}(\phi)$ converges to zero, then $Z_\lambda^{(2)}(\phi)$ converges in probability to zero and the CLT for $Z_\lambda(\phi)$ follows from Slutsky's Theorem.\jump

In the case $d=2$, the CLT for $Z_\lambda^{(4)}(\phi)$ follows directly from the Fourth Moment theorem, since $\Omega\in L^6$ implies that the contractions tends to zero. The variance of $Z_\lambda^{(6+)}(\phi)$ converges to a constant by Theorem \ref{thm6}, and so does the variance of $Z_\lambda(\phi) - Z_\lambda^{(4)}(\phi)$. The CLT for $Z_\lambda(\phi)$ then follows from Slutsky's Theorem.\jump

\noindent \underline{$\bullet$ Assume that $Y$ is the gradient of a real isotropic random wave}: \jump

The proof is very similar to the previous case with minor adaptations. Let $\mu$ be the spectral measure of the process $(\nabla F,TF)$ taking values in the space of Hermitian forms on $W = V\otimes(S^2(V^*))$.  From Lemma \ref{lemma23}, the spectral measure $\mu$ has a density w.r.t the spherical measure $\sigma$ given by
\[\dd \mu = \left[2i\pi\xi,-\frac{4\pi^2}{\sqrt{2\beta}}\left(\xi\otimes\xi-\frac{\gamma}{d}\|\xi\|^2\Id\right)\right]\otimes \left[-2i\pi\xi, -\frac{4\pi^2}{\sqrt{2\beta}}\left(\xi\otimes\xi-\frac{\gamma}{d}\|\xi\|^2\Id\right)\right]\dd \sigma.\]
Thanks to Lemma \ref{lemma16}, the measure $\mu*\mu$ has a density $\Psi$ w.r.t the measure $\sigma*\sigma$, which is bounded and polar-continuous. As in the previous paragraph, we can conclude that the limit constant is positive if we are able to show that for $\theta\in S^{d-1}$ the quantity $\Tr(f_4\Psi(0,\theta)f_4\Psi(0,\theta))$ is positive. According to Lemma \ref{lemma16},
\[\Psi(0,\theta) = \int \left(\left[2i\pi\xi,-\frac{4\pi^2}{\sqrt{2\beta}}\left(\xi\otimes\xi-\frac{\gamma}{d}\|\xi\|^2\Id\right)\right]\!\otimes\! \left[-2i\pi\xi, -\frac{4\pi^2}{\sqrt{2\beta}}\left(\xi\otimes\xi-\frac{\gamma}{d}\|\xi\|^2\Id\right)\right]\right)^{\otimes 2}\!\!\!\!\!\!\!\dd \sigma_\theta(\xi),\numberthis\label{eq:24}\]
where $\sigma_\theta$ is the uniform probability measure on the $d-2$-sphere on $V$ with normal vector $\theta$. The map $\Psi(0,\theta)$ can be seen as symmetric bilinear map over $S^2(W^*)$, as the fourth chaotic component $f_4$, see the paragraph above Lemma \ref{lemma27}. In particular, it follows from the symmetry of the integrand in Equation \eqref{eq:24} that the component associated to the subspace $V\otimes S^2(V)$ of $S^2(W^*)$ defined in Equation \eqref{eq:21} is stable by $\Psi(0,\theta)$. The restriction of $\Psi(0,\theta)$ to this subspace, denoted $\psi^r(0,\theta)$, is given by
\[\Psi^r(0,\theta) = \frac{(4\pi)^3}{2\beta}\int_{\Sp^{d-1}} \left[\xi\otimes\left(\xi\otimes\xi-\frac{\gamma}{4\pi^2}\Id\right)\right]^{\otimes 2} \dd \sigma_\theta(\xi).\]
We then have
\begin{align*}
&\Tr(f_4\Psi(0,\theta)f_4\Psi(0,\theta)) \geq \langle f_4^r\Psi^r(0,\theta)f_4^r\Psi^r(0,\theta))\\
&\geq \left(\frac{\alpha}{d}\right)^2(4\pi^2)^3\int_{\Sp^{d-1}}\int_{\Sp^{d-1}}f_4\left(\xi\otimes\left(\xi\otimes\xi-\frac{\gamma}{d}\|\xi\|^2\Id\right),\xi\otimes\left(\eta\otimes\eta-\frac{\gamma}{d}\|\xi\|^2\Id\right)\right)\dd \sigma_\theta(\xi)\dd \sigma_\theta(\eta).
\end{align*}
The integrand is continuous, and when $\xi=\eta$ its value is given, thanks to the expression of $f_4^r$ given by Lemma \ref{lemma27}, by
\begin{align*}
(\|\xi\|^2)^2\left[A\Tr\left(\left(\xi\otimes\xi-\frac{\gamma}{4\pi^2}\Id\right)^2\right) + B\Tr\left(\xi\otimes\xi-\frac{\gamma}{4\pi^2}\Id\right)^2\right]^2 = \|\xi\|^4\left(\frac{2\beta_0}{d}\alpha\right)^2,
\end{align*}
where the last equality follows from the same steps as in Equation \eqref{eq:22}. We conclude that the whole integral is positive, and so is the limit variance. The asymptotic normality is in all point similar to the previous paragraph.
\begin{acknowledgements}
\noindent I wish to thank Giovanni Peccati and Michele Stecconi for the numerous discussions about Berry cancellation, and their careful reading of the paper which greatly helped to improve the exposition of the results.
\end{acknowledgements}

\printbibliography
\end{document}